\documentclass[12pt]{article}
\usepackage{fancyhdr, array,calc,graphicx,url,tabularx}
\usepackage{geometry}
\usepackage{latexsym}
\usepackage{amssymb}
\usepackage{amsmath}
\usepackage{makeidx}
\usepackage{undertilde}

\makeindex

\renewcommand{\gg}{\gamma}

\newcommand{\rest}{\restriction}
\newcommand{\la}{\langle}
\newcommand{\ra}{\rangle}

\newcommand{\card}[1]{{\vert #1 \vert} }

\newcommand{\forces}{\Vdash}

\renewcommand{\models}{\vDash}

\newcommand{\cp}{{\rm cp }}

\newcommand{\cf}{{\rm cf}}

\newtheorem{theorem}{Theorem}[section]
\newtheorem{proposition}[theorem]{Proposition}
\newtheorem{definition}[theorem]{Definition}

\newtheorem{lemma}[theorem]{Lemma}
\newtheorem{corollary}[theorem]{Corollary}

\numberwithin{figure}{section}

\newenvironment{proof}{{\it{
Proof.}}}{\nopagebreak\mbox{}{\hfill$\square$}
\par\smallskip}
\newenvironment{Claim}{\noindent{\it
Claim} {}}{}
\newenvironment{Case}{\noindent{\it
Case}}{}

\newcommand{\rprop}[1]{Proposition~\ref{#1}}
\newcommand{\rthm}[1]{Theorem~\ref{#1}}
\newcommand{\rlem}[1]{Lemma~\ref{#1}}

\newcommand{\rdef}[1]{Definition~\ref{#1}}

\newcommand{\rsec}[1]{Section~\ref{#1}}

\def\k{\kappa}
\def\a{\alpha}
\def\b{\beta}
\def\d{\delta}

\def\l{\lambda}

\def\P{{\mathcal{P} }}
\def\W{{\mathcal{W} }}
\def\Q{{\mathcal{ Q}}}

\def\R{{\mathcal R}}

\def\H{{\rm{HOD}}}
\def\M{{\mathcal{M}}}
\def\N{{\mathcal{N}}}

\def\T {{\mathcal{T}}}
\def\U{{\mathcal{U}}}
\def\S{{\mathcal{S}}}

\def\VT{{\vec{\mathcal{T}}}}
\def\VU{{\vec{\mathcal{U}}}}

\def\cp #1{{ crit  #1 }}
\def\card#1{\left|#1\right|}

\def\iff{\mathrel{\leftrightarrow}}

\def\and{\mathrel{\kern1pt\&\kern1pt}}

\def\<#1>{\langle\,#1\,\rangle}

 \input xy
 \xyoption{all}

\title{On the prewellorderings associated to the directed systems of mice.
}

\author{Grigor Sargsyan \thanks{This material is partially based upon work supported by the National Science Foundation under Grant No DMS-0902628.}\\
        Department of Mathematics\\
        University of California\\
        Los Angeles, California 90095 USA\\
        http://math.ucla.edu/$\sim$grigor\\
        grigor@math.ucla.edu}

\date{\today}

\begin{document}

\maketitle

\begin{abstract}
Working under $AD$, we investigate the length of prewellorderings given by the iterates of $\M_{2k+1}$, which is the minimal proper class mouse with $2k+1$ many Woodin cardinals. In particular, we answer some questions from \cite{Hjorth01} (the discussion of the questions appears in the last section of \cite{HjorthD}). 
\end{abstract}

\baselineskip=24pt

In recent years, there have been many interactions between inner model theory and descriptive set theory. While the connection between the two areas was established early on in 1960s, the bulk of modern interactions go back to the work of Martin, Steel and Woodin carried out in late 80s and early 90s. In particular, Steel's computation of $\H^{L(\mathbb{R})}$ below $\Theta$ (see \cite{SteelHod}), Woodin's subsequent computation of $\H^{L(\mathbb{R})}$ (see \cite{WoodinHod}) and  Woodin's computation of $\H^{L[x][g]}$ (largely unpublished) have been of crucial importance for the results that followed\footnote{Here, $x$ is a real and letting $\k$ be the least inaccessible of $L[x]$, $g\subseteq Coll(\omega, <\k)$ is $L[x]$-generic.}. 

In this paper, we investigate the prewellordering associated with the directed system generated by $\M_{2k+1}$ where $k\in \omega$. Our intended application is the computation of the sup of the lengths of $\Game^{2k+1}(\omega\cdot n-\utilde{\Pi}^1_1)$-prewellorderings. We show that the sup is $\kappa^1_{2k+3}$. This generalizes Hjorth's computation of $\Game^{1}(\omega\cdot n-\utilde{\Pi}^1_1)$-prewellorderings. See \rsec{the main theorem} for the statement of the main theorem of this paper.
All the descriptive set theoretic notions that we will need come from \cite{Moschovakis} and and the inner model theoretic notions come from \cite{Zeman}. 

\textbf{Acknowledgments.} The results of this paper were proven in Berlin during the Spring of 2006 while the author was visiting his advisor John Steel. I am grateful to John Steel for introducing me to inner model theory and for bringing the questions considered in this paper to my attention. I also thank Farmer Schlutzenberg for very motivational conversations during Fall of 2006.  Finally, I express my deepest gratitude to the referee for providing long list of fundamental improvements. 

\section{On descriptive set theory}

We assume $AD$ throughout this paper. As is customary with descriptive set theorists, we let $\mathbb{R}$ be the Baire space $\omega^\omega$. We let $u_n$ be the $n$th uniform indiscernible and $s_n=\la u_i: i\leq n\ra$. We let $s_0=\emptyset$. Under $AD$, $u_n=\aleph_n$ (see \cite{Kanamori}).

Recall that for $x\in \mathbb{R}$, 
\begin{center}
$C_{2n}(x)=\{ y\in \mathbb{R}: y$ is $\Delta^1_{2n}(x)$ in a countable ordinal $\}$
\end{center}
and
\begin{center}
$Q_{2n+1}(x)=\{ y\in \mathbb{R}: y$ is $\Delta^1_{2n+1}(x)$ in a countable ordinal $\}$.
\end{center}
The definitions of $C_{2n}$ and $Q_{2n+1}$ given above are actually theorems as these are not the original definitions of these objects. The first equality is due to Harrington and Kechris (see \cite{HarKech}) and the second one is due to Kechris, Martin and Solovay (see \cite{Qtheory}). 


Following \cite{Moschovakis}, we let \textit{pointclass} stand for any collection of sets of reals (that is, we are not requiring closure under the set theoretic operations). If $\Gamma$ is a pointclass then $\breve{\Gamma}$ is the dual pointclass and $\Delta_\Gamma=\Gamma\cap \breve{\Gamma}$.

A relation $\leq$ is a \textit{prewellordering}\index{prewellordering} if it is transitive, reflexive, connected and wellfounded. Given a set of reals $A$, $\phi$ is a norm on $A$ if $\phi: A\rightarrow {\rm{Ord}}$. For each norm $\phi$ on $A$, we let $\leq^\phi$ be the binary relation on $A$ given by $x\leq^\phi y$ iff $\phi(x)\leq \phi(y)$. Then $\leq^\phi$ is a prewellordering of $A$. The opposite is true as well, given a prewellordering $\leq$ of $A$ there is an associated norm $\phi$ defined on $A$ such that $\leq=\leq^{\phi}$. If $\Gamma$ is a pointclass then $\phi$ is a $\Gamma$-norm if there are relations $\leq_{\Gamma}^\phi\in \Gamma$ and $\leq_{\breve{\Gamma}}^\phi\in \breve{\Gamma}$ such that for every $y\in dom(\phi)$ and for any $x\in \mathbb{R}$,
\begin{center}
$[x\in dom(\phi) \wedge \phi(x)\leq \phi(y)]\leftrightarrow x \leq^\phi_\Gamma y\leftrightarrow x\leq^\phi_{\breve{\Gamma}} y$.
\end{center}
If $\Gamma$ is a pointclass, we let
\begin{center}
$\d(\Gamma)=\sup\{ \leq^* : \leq^*\in \Gamma$ and $\leq^*$ is a prewellordering $\}$.
\end{center}

A sequence of norms $\vec{\phi}=\la \phi_i : i<\omega\ra$ on $A$ is a \textit{scale}\index{scale} on $A$ if whenever $\la x_i: i<\omega\ra\subseteq A$ is a sequence of reals converging to $x$ such that for each $i$ the sequence $\la \phi_i(x_k) : k<\omega\ra$ is eventually constant then $x\in A$ and for each $i$, $\phi_i(x)\leq \l_i$ where $\l_i$ is the eventual value of $\la \phi_i(x_k) : k<\omega\ra$. We write $x_i\rightarrow x (mod \vec{\phi})$ if $\la x_i: i<\omega\ra$ converges to $x$ in the above sense. $\vec{\phi}$ is a $\Gamma$-scale\index{$\Gamma$-scale} on $A$ if there are relations $R\in \Gamma$ and $S\in \breve{\Gamma}$ such that for all $y\in A$, for any $x\in \mathbb{R}$ and for any $n<\omega$

\begin{center}
$[x\in A \wedge \phi_n(x)\leq \phi_n(y)]\leftrightarrow R(n, x, y) \leftrightarrow S(n, x, y)$.
\end{center}

We say $\Gamma$ has the prewellordering property if every set in $\Gamma$ has a $\Gamma$-norm. We say $\Gamma$ has the scale property if every set in $\Gamma$ has a $\Gamma$-scale. For more on prewellordering property and scale property see \cite{Moschovakis}.

Suppose $\kappa$ is a cardinal. $T\subseteq \cup_{n<\omega}\omega^n\times\kappa^n$ is a tree if whenever $s\in T$ then $s\rest i\in T$ for any $i<lh(s)$. For $(x, f)\in \omega^\omega\times\kappa^\omega$ is a branch of $T$ if $(x\rest i, f\rest i)\in T$ for any $i<\omega$. $[T]$ is the set of branches of $T$. $p[T]$ is the projection of $[T]$ on the first coordinate, i.e., $x\in p[T]$ iff there is $f\in \kappa^\omega$ such that $(x, f)\in T$.

 A set of reals $A$ is $\kappa$-Suslin\index{$\kappa$-Suslin} if there is a tree $T\subseteq \cup_{n<\omega}\omega^n\times\kappa^n$ such that $A=p[T]$. $A$ is Suslin if it is $\kappa$-Suslin for some $\kappa$. Given a scale $\vec{\phi}$ on $A$ one can construct a tree $T$ such that $p[T]=A$. More precisely, let $T$ be the set of pairs $(s, f)$ such that there is some real $x\in A$ such that $s\lhd x$ and $f(i)=\phi_i(x)$ for each $i<lh(f)$. Given a tree $T$ such that $p[T]=A$, one can get a scale $\vec{\phi}$ on $A$ by considering the leftmost branches of $T$ (see \cite{Moschovakis}). Thus, carrying a scale and being Suslin are equivalent.

 Finally, we say that $\kappa$ is a \textit{Suslin cardinal}\index{Suslin cardinal} if there is a set of reals $A$ which is $\kappa$-Suslin but $A$ is not $\eta$-Suslin for any $\eta<\k$. We let $S(\kappa)$\index{$S(\k)$} be the pointclass of $\kappa$-Suslin sets. It is not hard to show that $S(\kappa)$ is closed under projections (see \cite{Moschovakis}). For more on trees and Suslin sets see \cite{Moschovakis}. For a complete characterization of Suslin cardinals see \cite{Jackson}.

Under $AD$, for each $n$ and real $z$, $\Pi^1_{2n+1}(z)$ and $\Sigma^1_{2n+2}(z)$ have the scale property. The sup of $\utilde{\Pi}^1_{2n+1}$ prewellorderings and $\utilde{\Sigma}_{2n+2}$ prewellorderings play an important role in descriptive set theory. Following \cite{Moschovakis}, we let
\begin{center}
$\d^1_{2n+1}=\d(\utilde{\Pi}^1_{2n+1})=\d(\Pi^1_{2n+1})$
\end{center}
and
\begin{center}
$\d^1_{2n}=\d(\utilde{\Sigma}^1_{2n})$.
\end{center}
It turns out that under $AD$,
\begin{center}
$\d^1_{2n}=(\d_{2n+1}^1)^+$
\end{center}
and $\d^1_{2n+1}$ is a successor cardinal whose predecessor is denoted by $\kappa^1_{2n+1}$ (see \cite{Moschovakis}). It is shown in \cite{Moschovakis} that
\begin{center}
$\utilde{\Sigma}^1_{2n+3}=S(\kappa^1_{2k+1})$.
\end{center}
Also, $\kappa^1_3=\aleph_\omega$, $\d^1_3=\aleph_{\omega+1}$ and $\d^1_4=\aleph_{\omega+2}$.

$\Game$ is the \textit{game quantifier}. Recall that given a set of reals $A\subseteq \mathbb{R}^2$ we let $\Game A$ be the set
\begin{center}
$x\in \Game A\iff \exists x_0\forall x_1\exists x_2 \forall x_3 \cdot \cdot \cdot \exists x_{2n} \forall x_{2n+1} \cdot \cdot \cdot ( ( x, \la x_i: i<\omega\ra) \in A)$.
\end{center}
Here, the quantifiers range over $\omega$. Equivalently,
\begin{center}
$\Game A =\{ x \in \mathbb{R} : $ player $I$ has a winning strategy in $G_{A_x}\}$.
\end{center}
where $A_x=\{ y: (x, y)\in A\}$. A set is $\omega\cdot n-\utilde{\Pi}^1_1$ if there is a sequence $\la A_\a : \a< \omega \cdot n\ra\subseteq \utilde{\Pi}^1_1$ such that
\begin{center}
$x\in A\iff$ the least $\a$ such that $x\not \in A_\a$ is odd.
\end{center}
Equivalently sets in $\omega\cdot n-\utilde{\Pi}^1_1$ constitute the first $\omega\cdot n$ levels of the \textit{difference hierarchy} for $\utilde{\Pi}^1_1$.

\section{On inner model theory}

Recall that if $\M$ is a premouse then $\mathcal{G}(\M, \k)$ is the two player iteration game that has $<\k$ moves (see \cite{OIMT}). In this game, player $I$ plays the successor steps which amounts to choosing an extender and applying it to the earliest model it makes sense to apply. Player II plays limit stages and her job is to choose a well-founded cofinal branch of the resulting iteration tree. II wins if all the models produced in the game are well founded. $\Sigma$ is then called a $\k$-iteration strategy for $\M$ if it is a winning strategy for player $II$. 

If $\M$ is a mouse\footnote{The reader should consult \cite{OIMT} for the definition of a mouse.} and $\xi\leq o(\M)$, then we let $\M||\xi$ be $\M$ cutoff at $\xi$, i.e., we keep the predicate indexed at $\xi$. We let $\M|\xi$ be $\M||\xi$ without the last predicate. We say $\xi$ is a cutpoint of $\M$ if there is no extender $E$ on $\M$ such that $\xi\in(\cp(E), lh(E)]$. We say $\xi$ is a strong cutpoint if there is no $E$ on $\M$ such that $\xi\in[\cp(E), lh(E)]$.

If $\T$ is an iteration tree, i.e., a play of the game, then, following the notation of \cite{FSIT}, $\T$ has the form
\begin{center}
$\T=\la T, deg, D, \la E_\a, \M^*_{\a+1}| \a+1<\eta\ra\ra$.
\end{center}
Recall that $D$ is the set of \textit{dropping} points. Recall also that if $\eta$ is limit then
\begin{center}
$\vec{E}(\T)= \cup_{\a<\eta}(\dot{E}^{\M_\a\rest  lh(E_\a)})$,\\
$\M(\T)=\cup_{\a<\eta}\M_\a\rest lh(E_\a)$,\\
$\d(\T)=\sup_{\a<\eta} lh(E_\a)$.
\end{center}
If $b$ is a branch of $\T$ then $\M^\T_b$ is the branch model of the tree. Then if $\a\leq_T\b$ then $i_{\a, \b}^\T:\M^*_\a\rightarrow \M_\b^\T$ is the iteration map if $[\a, \b]_\T\cap D=\emptyset$ and $i_{\a, b}^\T:\M^*_\a\rightarrow \M_b^\T$ is the iteration map if $\a\in b$ and $b-\a\cap D =\emptyset$. \textit{In this paper, all iteration trees are normal. We will refer to the general iterations as stacks of normal trees.} 

It is by now a standard fact that if $b$ and $c$ are cofinal branches of $\T$ on $\M$ and $\R=\M_b^\T\cap \M_c^\T$ then $\R\models ``\d(\T)$ is Woodin" (see \cite{OIMT}). Moreover, if $\Q$ is a mouse over $\M(\T)$ (this in particular means that $\Q$ has no extender overlapping with $\d(\T)$) such that $\Q\models ``\d(\T)$ is Woodin" yet there is a counterexample to Woodiness of $\d(\T)$ in $L_1(\Q)$ then there is at most one cofinal branch $b$ of $\T$ such that $\Q\trianglelefteq\M^\T_b$ (see \cite{OIMT}). The following lemma, which builds upon the proof of the aforementioned fact is one of the most important ingredients available to us and will be used in this paper many times. It is essentially due to Martin and Steel, see Theorem 2.2 of \cite{IT}.

\begin{lemma}[Uniqueness of branches]\label{uniqueness of branches} Suppose $\M$ is a mouse and $\T$ is an iteration tree on $\M$ of limit length. Suppose $s$ is a cofinal subset of $\d(\T)$. Then there is at most one cofinal branch $b$ such that there is $\a\in b$ with the property that $i^\T_{\a, b}$ exists and $s\subseteq ran(i^{\T}_{\a, b})$.
\end{lemma}
\begin{proof}
Towards a contradiction, suppose there are two cofinal branches $b$ and $c$ such that for some $\a, \b$, both $i^\T_{\a, b}$ and $i^\T_{\b, c}$ exist and $s\subseteq ran(i^\T_{\a, b})\cap ran(i^\T_{\b, c})$. Without loss of generality we can assume that $\a$ and $\b$ are the least ordinals with this property, $\a\leq \b$ and that $b$ and $c$ diverge at $\a$ or earlier, i.e., if $\gg$ is the least ordinal in $b\cap c$ then $\gg\leq \a$. By \cite{IT}, we can assume that $b$ is the downward closure of $\la \a_n: n<\omega\ra$,  $c$ is the downwards closure of $\la \b_n: n<\omega\ra$, $\a_0=\a$ and $\b_0=\b$. Let then $\xi$ be the least ordinal in $ran(i^\T_{\a, b})\cap ran(i^\T_{\b, c})$. Let $n$ be the least such that $\cp(i^\T_{\a_n, b})>\xi$. This means that $\cp(E^\T_{\a_{n+1}-1})>\xi$ and that $lh(E^\T_{\a_n})<\xi$. By the proof of Theorem 2.2 of \cite{IT}, this means that for some $m\geq 1$, $\xi \in [\cp(E^\T_{\b_m-1}), lh(E^\T_{\b_m-1}))$. This then implies that $\xi\not \in ran(i^\T_{\b_{m-1}, c})$, which is a contradiction.
\end{proof}

The proof of \rlem{uniqueness of branches} gives the following as well.

\begin{lemma}\label{partial agreement} Suppose $\M$ is a mouse and $\T$ is an iteration tree on $\M$ of limit length. Suppose $b, c$ are two cofinal branches of $\T$ such that $i^\T_b$ and $i^\T_c$ exist. Suppose that for some $\a$,
\begin{center}
$i^\T_b(\a)=i^\T_c(\a)<d(\T)$.
\end{center}
Then $i^\T_b\rest \a = i^\T_c\rest \a$. Moreover, if $\xi\in b$ is the least such that $\cp(E_\xi^\T)>i^\T_b(\a)$ then $b\cap \xi=c\cap \xi$.
\end{lemma}

If $\M$ is a mouse and $\T$ is a tree then we say $\T$ is above $\eta$ if all extender used in $\T$ have critical point $>\eta$. If $\Sigma$ is an $(\omega_1, \omega_1)$-iteration strategy for $\M$ and $\VT$ is a stack of trees on $\M$ according $\Sigma$ with last model $\N$ then we let $\Sigma_{\N, \VT}$ be the strategy of $\N$ induced by $\Sigma$. We say $\Sigma$ has the Dodd-Jensen property if whenever $\N$ is an iterate of $\M$ via $\Sigma$ and $\pi:\M\rightarrow \W\trianglelefteq \N$ is (fine structural) embedding then the iteration from $\M$ to $\N$ doesn't drop, $\W=\N$ and if $i:\M\rightarrow \N$ is the iteration embedding then for every $\a$, $i(\a)\leq \pi(\a)$. If $\Sigma$ has the Dood-Jensen property and $\VT$ and $\VU$ are two stacks on $\M$ with last model $\N$ such that $i^\VT$ and $i^{\VU}$ exist then $i^{\VT}=i^{\VU}$ and $\Sigma_{\N, \VT}=\Sigma_{\N, \VU}$. Lastly, we let
\begin{center}
$I(\M, \Sigma)=\{ \N :$ there is a stack $\VT$ on $\M$ according to $\Sigma$ with last model $\N$ and $i^{\VT}$ exists $\}$.
\end{center}

\subsection{$\S$-constructions}

Here we introduce $\S$-constructions which were first introduced in \cite{Selfiter} where they were called $P$-constructions. Such constructions are due to Steel and hence, we change the terminology and call them $\S$-constructions. These constructions allow one to translate mice over some set $A$ to mice over some set $B$ provided $A$ and $B$ are somehow close. The complete proof of the following proposition is essentially the proof of Lemma 1.5 of \cite{Selfiter}.

\begin{proposition}\label{s-constructions prop}
Suppose $\M$ is a sound mouse and $\d$ is a strong cutpoint cardinal of $\M$.  Suppose further that $\N\in \M|\d+1$ is such that $\d\subseteq \N\subseteq H_\d^\M$ and there is a partial ordering $\mathbb{P}\in L_{\omega}[\N]$ such that whenever $\Q$ is a mouse over $\N$ such that $H_\d^\Q=\N$ then $\M|\d$ is $\mathbb{P}$-generic over $\Q$. Then there is a mouse $\S$ over $\N$ such that $\M|\d$ is generic over $\S$ and $\S[\M|\d]=\M$.
\end{proposition}

It is clear what $\S$ must be. Because $\mathbb{P}$ is a small forcing with respect to the critical points of the extenders of $\M$ that have indices bigger than $\d$, all such extenders can be put on a sequence of some mouse over $\N$. This is exactly what $S$-constructions do. An $S$-construction of $\M$ over $\N$ is a sequence of $\N$-mice $\la \S_\a, \bar{\S}_\a: \a\leq \eta\ra$ such that
\begin{enumerate}
\item $\S_0=L_{\omega}[\N]$,
\item if $\M|\d$ is generic over $\bar{\S}_\a$ for a forcing in $L_{\omega}[\N]$ then $\bar{\S}_\a[\N]=\M|(\omega\times\a)$ and
\begin{enumerate}
\item if $\M||(\omega\times \a)$ is active then $\S_\a$ is the expansion of $\bar{\S_\a}$ by the last extender of $\M||(\omega\times \a)$ and $\bar{\S}_{\a+1}=rud(\S_\a)$,
\item if $\M||(\omega\times \a)$ is passive then $\S_\a=\bar{\S_\a}$ and $\bar{\S}_{\a+1}=rud(\S_\a)$,
\end{enumerate}
\item if $\l$ is limit then $\bar{S}_\l=\cup_{\a<\l}\S_\a$.
\end{enumerate}

By the proof of Lemma 1.5 of \cite{Selfiter}, the $\S$-construction described in 1-3 cannot fail as long as the hypothesis of 2 holds. Thus, we always have a last model of $\S$-construction which might be some $\bar{\S}_\a$ instead of $\S_\a$. 

\begin{definition}\label{s(n)} We let $\S^\M(\N)$ be the last model of the $\S$ construction done over $\N$.
\end{definition}

Then by the proof of Lemma 1.5 of \cite{Selfiter}, $\S[\M|\d]\trianglelefteq\M$. Moreover, if the hypothesis of 2 never fails then in fact, $\S[\M|\d]=\M$. It also follows that $\S$ inherits whatever iterability $\M$ has above $\d$. The method of $\S$-constructions is a very useful inner model theoretic tool. A particularly important application for us is the following lemma.

\begin{lemma}\label{s-constructions give q-structures} Suppose $\M\models ZFC-Powerset$ is a mouse and $\eta$ is a strong cutpoint non-Woodin cardinal of $\M$. Suppose $\gg>\eta$ is a cardinal of $\M$ and $\N=L[\vec{E}]^{\M|\gg}$. Suppose $L_{\omega}(\N|\eta)\models ``\eta$ is Woodin". Let $\la \S_\a, \bar{\S}_\a :\a<\nu\ra$ be the $\S$-construction of $\M|(\eta^+)^\M$ over $\N|\eta$. Then for some $\a<\nu$, $\S_\a\models ``\eta$ isn't Woodin".
\end{lemma}
\begin{proof}
Let $\S$ be the last model of the $\S$-construction of $\M|(\eta^+)^\M$ over $\N|\eta$. Suppose $\eta$ is a Woodin cardinal of $\S$. Then $\M|\eta$ is generic for the $\eta$-generator version of the extender algebra of $L_{\omega}(\N|\eta)$. we also have that $\M|\eta$ is generic over $\S$ for the $\eta$-generator version of the extender algebra at $\eta$ and hence, $\S[\M|\eta]=\M|(\eta^+)^\M$. Thus, $\eta$ isn't Woodin in $\S[\M|\eta]$. Let $f:\eta\rightarrow \eta$ be the function in $\M$ witnessing that $\eta$ isn't Woodin. Then because the $\eta$-generator version of extender algebra is $\eta$-cc, there is $g\in \S$ which dominates $f$. Let $E\in \vec{E}^\S$ be the extender that witnesses that $\eta$ is Woodin for $g$. Then if $E^*$ is the background extender of $E$ then $E^*$ witnesses the Woodiness of $\eta$ for $f$ in $\M$, contradiction!
\end{proof}

Before moving on, we set up one last notation. Given a model $M$ of a fragment of $ZFC$ with a unique Woodin cardinal, we let $\mathbb{B}^M$ be the extender algebra of $M$ at its unique Woodin cardinal. If $G\subseteq \mathbb{B}^M$ then we let $x_G$ be the set naturally coded by $G$.

\section{Descriptive inner model theory}\label{dimt}

We let $\M_n$ be the minimal proper class mouse with $n$ Woodin cardinals. $\M_n^\#$ is the minimal mouse with last extender and with $n$ Woodin cardinals. Clearly, $\M_n$ is the result of iterating the last measure of $\M_n^\#$ through the ordinals. We let $\M_0=L$. In \cite{PWOIM}, Steel and Woodin computed the descriptive set theoretic complexity of the reals of $\M_n$. They showed that
\begin{center}
$C_{2n+2}(x)=\mathbb{R}^{\M_{2n}(x)}$
\end{center}
and
\begin{center}
$Q_{2n+3}(x)=\mathbb{R}^{\M_{2n+1}(x)}$.
\end{center}
We let
 \begin{displaymath}
   S_n(x) = \left\{
     \begin{array}{lr}
       C_{n+2}(x) :&  n \ is\ even \\
       Q_{n+2}(x) :&  n\ is\ odd \
     \end{array}
   \right.
\end{displaymath}
It is then clear that
\begin{center}
$S_n(x)=\mathbb{R}^{\M_n(x)}$.
\end{center}
Using standard techniques, we can now define $S_n(a)$ for any countable set $a$. More precisely, $b\in S_n(a)$ if for comeager many $g\subseteq Coll(\omega, a)$ letting $x_g$ be the real coding $a$ and $y_g$ be the real coding $b$ then $y_g\in S_n(x_g)$. 

We also let $\M_\omega$ be the minimal proper class mouse with $\omega$ Woodin cardinals and $\M_\omega^\#$ be the minimal mouse with $\omega$ Woodin cardinals and with a last extender. Then $\M_\omega$ is the result of iterating the last measure of $\M_\omega^\#$ through the ordinals. The following theorems ara what allow us to use inner model theoretic tools to investigate descriptive set theoretic objects. The proofs of these results can be found in \cite{OIMT}. 

\begin{theorem}[Woodin] Suppose $\M_\omega^\#$ exists and is $\omega_1$-iterable. Then $AD^{L(\mathbb{R})}$ holds. 
\end{theorem} 
\begin{theorem}[Steel-Woodin]\label{mouse capturing} Suppose $\M_\omega^\#$ exists and is $\omega_1$-iterable. Let $\Gamma=(\Sigma^2_1)^{L(\mathbb{R})}$. Then for every countable transitive set $a$, 
\begin{center}
$C_{\Gamma}(a)=\mathbb{R}^{\M_\omega(a)}=\cup \{\mathbb{R}^\N: L(\mathbb{R})\models \N$ is a sound $\omega_1$-iterable $a$-mouse such that for some $n<\omega$, $\rho_n(\N)=a\}$.
\end{center}
\end{theorem}  

Let $\Sigma$ be the canonical iteration strategy of $\M_\omega$. Let
\begin{center}
$\mathcal{F}=\{ \P:$ there is a $\Sigma$-iterate $\N$ of $\M_\omega$ such that $\P=\N|(\nu^{+\omega})^\N$ where $\nu$ is a successor cardinal of $\N$ which is less than the least $\N$-cardinal which is strong to the least Woodin of $\N\}$.
\end{center}
Then it follows from \rthm{mouse capturing} that for every $\P\in \mathcal{F}$, $\Sigma_\P\in L(\mathbb{R})$. To see this, notice that whenever $\T$ is a tree on $\P$ of limit length and $b$ is a well-founded branch then $\Q(b, \T)$ exists. Now, if $\T$ is according to $\Sigma_\P$ and $b=\Sigma_\P(\T)$ then it follows from \rthm{mouse capturing} that $\Q(b, \T)$ has an iteration strategy in $L(\mathbb{R})$. Thus, $L(\mathbb{R})$ can uniquely identify $b$. The details of such arguments appear in Section 7 of \cite{OIMT}.

We can define $\leq_{\mathcal{F}}$ on $\mathcal{F}$ by $\P\leq_{\mathcal{F}} \Q$ iff there is $\a$ such that $\Q|\a\in I(\P, \Sigma_\P)$. Notice that if $\P\leq_{\mathcal{F}}\Q$ and $\a$ is such that $\Q|\a\in I(\P, \Sigma_\P)$ then for some $\nu<\a$, $\a=(\nu^{+})^\Q$. If $\P\leq_{\mathcal{F}}\Q$ then we let $i_{\P, \Q}:\P\rightarrow \Q|\alpha$ be the iteration embedding.

Notice that $\leq_{\mathcal{F}}$ is directed and hence, we can let $\M_\infty$ be the direct limit of $(\mathcal{F}, \leq_{\mathcal{F}})$. We then have that

\begin{theorem}[Steel, \cite{SteelHod}]\label{steel's thing} $L(\mathbb{R})\models \M_\infty=V_{\delta}^\H$ 
where $\d=\d(\utilde{\Sigma}^2_1)$. 
\end{theorem}

Woodin extended this result to compute the full $\H$ of $L(\mathbb{R})$. We refer the reader to \cite{WoodinHod} for more on Woodin's work on $\H^{L(\mathbb{R})}$. It is important to note that the existence of $\M_\omega^\#$, which is a tiny bit stronger than $AD^{L(\mathbb{R})}$, is unnecessary and all the results in this paper can be proved only from $AD^{L(\mathbb{R})}$. Nevertheless, it is convenient and aesthetically more pleasant to assume that $\M_\omega^\#$ exists and we will do so whenever we wish. Experts will have no problem seeing how to remove this assumption. We refer the reader to \cite{OIMT} for an expanded version of this short summary of inner model theory. \cite{OIMT} also proves most of the results stated in this section without assuming the existence of $\M_\omega^\#$ but just $AD^{L(\mathbb{R})}$.

\section{The main theorem}\label{the main theorem}

By a result of Martin (see \cite{Martin83}) and Neeman (see \cite{Neeman95}), for $k\geq 1$, a set of reals $A$ is $\Game^{k}(\omega\cdot n-\utilde{\Pi}^1_1)$ iff there is $m\in \omega$, a real $z$ and a formula $\phi$ such that
\begin{center}
$x\in A\iff \M_{k-1}(x, z)\models \phi[x, z, s_m]$.
\end{center}
We let $\Gamma_{k, m}(z)$ be the set of reals $A$ such that there is a formula $\phi$ such that, letting $s_m$ be the sequence of the first $m$ uniform indiscernibles, 
\begin{center}
$x\in A\iff \M_{k-1}(x, z)\models \phi[x, z, s_m]$.
\end{center}
We let $\Gamma_{k, m}=\Gamma_{k, m}(0)$ and $\utilde{\Gamma}_{k, m}=\cup_{z\in \mathbb{R}}\Gamma_{k, m}(z)$. Also, we let $\utilde{\Gamma}_k=\cup_{m<\omega}\utilde{\Gamma}_{k, m}$.

In \cite{Hjorth01}, Hjorth computed the sup of the lengths of $\utilde{\Gamma}_{1, m}$-prewellorderings. He showed that
\begin{center}
$\d(\utilde{\Gamma}_{1, m})\leq u_{m+2}$.
\end{center}
and therefore,
\begin{center}
$\kappa^1_3=\aleph_\omega=\d(\utilde{\Gamma}_{1})$\footnote{It is not hard to see that the standard prewellordering of the $\{x^\# : x\in \mathbb{R}\}$ has length $\kappa^1_3$, i.e, let $\phi(n, m, x^\#)=\tau_n^{L[x]}(x, s_m)$ where $\la \tau_n:n<\omega\ra$ is some enumeration of the terms in the appropriate language. Then $\phi$ has length $u_\omega=\k^1_3$ and for each $m$ letting $\phi_m$ be the prewellordering given by $\phi_m(n, x^\#)=\phi(n, m. x^\#)$, we have that $\phi_m\in \Gamma_{1, m+1}$. Thus, we indeed have an equality.}.
\end{center}

In this paper, assuming $AD$, we compute $\d(\utilde{\Gamma}_{k, m})$. First let
\begin{center}
$a_{k,m}=\d(\utilde{\Gamma}_{k, m})$.
\end{center}
Here is our main theorem.

\begin{theorem}[Main Theorem]\label{main theorem} Assume $AD$ and let $k$ be an integer. Then
\begin{center}
$\sup_{m<\omega} a_{2k+1, m} =\k^1_{2k+3}$.
\end{center}
\end{theorem}

We will prove the theorem using directed systems of mice. Our proof relies on a generalization of Woodin's analysis of $\H^{L[x][g]}$. The proof is divided into subsections. The proof presented here suggests further applications of the directed systems in descriptive set theory and we will end with a discussion of projects that are left open. We start with introducing the direct limit associated with $\M_{n}$'s.

\section{The directed system associated to $\M_{n}$.}

In this section, we analyze the length of the prewellordering given by the iterates of $\M_{2n+1}$. As it turns out, the even case, i.e., the prewellordering associated to $\M_{2n}$s, doesn't give much beyond the results of \cite{OIMT}. Nevertheless, we make all the definitions for arbitrary $n$. The prewellordering associated with the iterates of $\M_{n+1}$ that we are interested in is the following.

For any iterate $\P$ of $\M_{n+1}$ we let $\d^\P$ be the least Woodin of $\P$. Let $\Sigma$ be the canonical iteration strategy of $\M_{n+1}$. If $\P\in I(\M_{n+1}, \Sigma)$ and $\Q\in I(\P, \Sigma_\P)$ then we let $i_{\P, \Q}$ be the iteration embedding. We then define a prewellordering $R^+_n$ of the set
\begin{center}
$\{ (\P, \a) : \P\in I(\M_{n+1}, \Sigma) \wedge \a<\d^\P\}$
\end{center}
by $(\P, \a)R^+_n (\Q, \b)$ iff $\Q\in I(\P, \Sigma_\P)$ and $i_{\P, \Q}(\a)\leq \b$. Clearly $R^+_n$ is a prewellordering. One problem with $R^+_n$ is that it is a prewellordering of uncountable objects and hence, cannot be regarded as a prewellordering of the reals. Here is how one can find an equivalent prewellordering of countable objects.

We let $\W_n=\M_{n+1}|(\d^{+\omega})^{\M_{n+1}}$ and define the equivalent of $R^+_n$ on the set
\begin{center}
$\mathcal{J}_n^+=\{(P, \a): \P\in I(\W_n, \Sigma_{\W_n})\wedge\a<\d^\P\}$.
\end{center}
We set $(\P, \a)R_n^+ (\Q, \b)$ iff $\Q \in I(\P, \Sigma_{\P})$ and $i_{\P, \Q}(\a)\leq \b$. It is not hard to see that $R^+_n$ is essentially the old $R^+_n$. Two questions then immediately come up: 1. What is the length of $R_n^+$? and 2. What is the complexity of $R_n^+$? It is not hard to find an upper bound for $R_n^+$.

\begin{lemma} $\card{R^+_n}<\d^1_{n+3}$.
\end{lemma}
\begin{proof}
Here is the outline of the proof. Because $x\rightarrow \M_n^\#(x)$ is a $\Pi^1_{n+2}$ (see \cite{PWOIM}), we get that $\mathcal{J}_n^+$ is $\Sigma^1_{n+3}(\W_n)$ (i.e., $\Sigma^1_{n+3}(x)$ for any code $x$ of $\W_n$). The complexity essentially comes from the fact that we require $i_{\P, \Q}$ be the correct iteration embedding and to say that we need to refer to $x\rightarrow \M_n^\#(x)$ operator. 
\end{proof}

To prove our main theorem we need to somehow internalize $R^+_n$ to $\M_n(x)$ where $x$ is any real coding $\W_n$. Notice that $\M_n(x)$ doesn't know the strategy of $\W_n$ and hence, it doesn't know how to define its own version of $R^+_n$. We will define an enlargement of $R^+_n$ which $\M_n(x)$ can define and we will show that the enlargement has the same length as $R^+_n$. 

We now start introducing concepts that we will need in order to internalize $R^+_n$ to $\M_n(x)$. Most of these concepts have their origins in Woodin's unpublished work on $\H^{L[x][g]}$. Various sources have expositions of similar concepts. For example, \cite{WoodinHod} has most of what we need excepts for the full hod limit. None of these concepts appeared for projective mice such as $\M_n$ and here we take a moment to develop these ideas. We start with suitability. First recall the $S_n$ operator from \rsec{dimt}.

\begin{definition}[$n$-suitable]\label{$n$-suitable}
$\P$ is \emph{$n$-suitable} if there is $\d$ such that
\begin{enumerate}
\item $\P\models ZFC-Replacement$,
\item $\P\models ``\d$ is the only Woodin cardinal",
\item $o(\P)=\sup_{i<\omega}(\d^{+i})^\P$,
\item for every strong cutpoint cardinal $\eta$ of $\P$, $S_n(\P|\eta)=\P|(\eta^+)^\P$.
\end{enumerate}
\end{definition}

If $\P$ is $n$-suitable then we let $\d^\P$ be the $\d$ of \rdef{$n$-suitable}. Clearly $\W_n$ is a $n$-suitable premouse. Moreover, if $\Q\in I(\W_n, \Sigma_{\W_n})$ then $\Q$ is $n$-suitable because $i_{\W_n, \Q}$ can be lifted to $i:\M_{n+1}\rightarrow \M_n(\Q)$. Sometimes we will just say that $\P$ is $n$-suitable implying that it is $n$-suitable for some $n$.

To approximate the iteration strategy of $\W_n$ inside $\M_n(x)$, the notion of $s$-iterability is used. We now work towards introducing it.
Given an iteration tree $\T$ on an $n$-suitable $\P$, we say $\T$ is \textit{correctly guided} if for every limit $\a<lh(\T)$, if $b$ is the branch of $\T\rest \a$ chosen by $\T$ and $\Q(b, \T\rest \a)$ exists then $\Q(b, \T\rest \a)\trianglelefteq \M_n(\M(\T\rest \a))$. $\T$ is \textit{short} if there is a well-founded branch $b$ such that $\T^\frown \{\M^\T_b\}$ is correctly guided. $\T$ is \textit{maximal} if $\T$ is not short.

Suppose $\P$ is $n$-suitable. We say $\la \T_i, \P_i : i<m\ra$ is a \textit{finite correctly guided stack} on $\P$ if
\begin{enumerate}
\item $\P_0=\P$,
\item $\P_i$ is $n$-suitable and $\T_i$ is a correctly guided tree on $\P_i$ below $\d^{\P_i}$,
\item for every $i$ such that $i+1< m$ either $\T_i$ has a last model and $i^\T$-exists or $\T$ is maximal, and
    \begin{enumerate}
    \item if $\T_i$ has a last model then $\P_{i+1}$ is the last model of $\T_i$,
    \item if $\T_i$ is maximal then $\P_{i+1}=\M_n(\M(\T_i))|(\d(\T_i)^{+\omega})^{\M_n(\M(\T_i))}$.
    \end{enumerate}
\end{enumerate}
We say $\Q$ is the last model of $\la \T_j, \P_j : i<k\ra$ if one of the following holds:
 \begin{enumerate}
 \item $\T_{k-1}$ has a last model and $\Q$ is the last model of $\T_{k-1}$,
 \item $\T_{k-1}$ is short and there is a cofinal well-founded branch $b$ such that $\Q(b, \T)$ exists and is iterable and $\Q=\M^\T_b$,
 \item $\T_{k-1}$ is maximal and
 \begin{center}
$\Q=\M_n(\M(\T_{k-1}))|(\d(\T_{k-1})^{+\omega})^{\M_n(\M(\T_{k-1}))}$.
\end{center}
 \end{enumerate}

We say $\Q$ is a \textit{correct iterate} of $\P$ if there is a correctly guided finite stack on $\P$ with last model $\Q$.

Suppose $\P$ is $n$-suitable and $s=\la \a_0, ...,\a_m\ra$ is a finite sequence of ordinals. Then we let $T^\P_{s, k}\subseteq [((\d^\P)^{+k})^\P]^{<\omega}\times \omega$ be the set
\begin{center}
$(t, \phi)\in T^\P_{s, k} \iff \phi$ is $\Sigma_1$ and $\M_n(\P)\models \phi[t, s]$.
\end{center}
\begin{center}
$\gg^\P_{s}=Hull^\P_1( \{ T^\P_{s, i} : i\in \omega \} )\cap \d^\P$.
\end{center}
Notice that
\begin{center}
$\gg^\P_{s}=Hull^\P_1(\gg_s^\P\cup \{ T^\P_{s, i} : i\in \omega\})\cap \d^\P$.
\end{center}
Let
\begin{center}
$H_{s}^\P =Hull^\P_1(\gg^\P_{s}\cup \{ T^\P_{s, i} : i\in \omega \})$.
\end{center}

If $s=s_m$, then we let $\gg_{m}^\P=\gg_{s_m}^\P$ and $H_{m}^\P=H_{s_m}^\P$.
The following is not hard to show.

\begin{lemma}\label{ggs sup up}
$\sup_{n<\omega}\gg_n^\P=\d^\P$.
\end{lemma}
\begin{proof}
Suppose not. Let $\gg=\sup_{n<\omega}\gg_n^\P$. Let $X=Hull_1^\P(\gg\cap \{ T_{s_m, i}^\P : m, i\in \omega\})$. Let $\N$ be the collapse of $X$ and let $\pi:\N\rightarrow \P$ be the inverse of the collapsing map. We have that for each $m, i$ there is $S_{m, i}\in \N$ such that $\pi(S_{m, i})=T_{s_m, i}^\P$. We have that $\gg=\d^\S$. Notice that for each $i$, $\cup_{m\in \omega} S_{m, i}$ is a complete and consistent theory and if $\R$ is its model then $\R$ is essentially the hull of ordinals $<(\gg^{+i})^\N$ and $\omega$ many indiscernibles. Moreover,  we have that $\pi$ can be extended to $\pi^*:\R \rightarrow \M_n(\P)$. This implies that $\R$ is well-founded and therefore, it has to be $\M_n(\N|(\gg^{+i})^\N)$. This shows that $\M_n(\N|\gg)\models ``\gg$ is Woodin" which implies that $\P\models ``\gg$ is Woodin". This is a contradiction as $\d^\P$ is the least Woodin of $\P$.
\end{proof}

\begin{definition}[$s$-iterability]\label{s-iterability} Suppose $\P$ is $n$-suitable and $s=\la \a_i :i< l\ra$ is an increasing finite sequence of ordinals. $\P$ is \emph{$s$-iterable} if whenever $\la \T_k, \P_k : k<m\ra$ is a finite correctly guided stack on $\P$ with last model $\Q$
then there is a sequence $\la b_k: k< m\ra$ such that
\begin{enumerate}
\item for $k<m-1$,
 \begin{displaymath}
   b_k = \left\{
     \begin{array}{lr}
      \emptyset :&  \T_k \ has \ a \ successor \ length \\
      cofinal\ well-founded\\ branch\ such\ that\ \M^\T_{b_k}=\P_\k :& \T_k \ is \ maximal
     \end{array}
   \right.
\end{displaymath}

\item if $\T_{m-1}$ has a successor length then $b_{m-1}=\emptyset$, if $\T_{m-1}$ is short then $b_{m-1}$ is the unique cofinal well-founded branch such that $\Q(b_{m-1}, \T_{m-1})$ exists and is iterable, and if $\T_{m-1}$ is maximal then $b_{m-1}$ is a cofinal well-founded branch,
\item letting
 \begin{displaymath}
   \pi_k = \left\{
     \begin{array}{lr}
      i^{\T_k} :&  \T_k \ has \ a \ successor \ length \\
      i_{b_k}^{\T_k} :& \T_k \ is \ maximal
     \end{array}
   \right.
\end{displaymath}
and $\pi=\pi_{m-1}\circ \pi_{m-2}\circ\cdot \cdot\cdot \pi_0$ then for every $l$
    \begin{center}
  $\pi(T^\P_{s, l})=T^\Q_{s, l}$.
    \end{center}
\end{enumerate}
\end{definition}

Suppose $\P$ is $n$-suitable, $s=\la \a_i :i< l\ra$ is an increasing finite sequence of ordinals and $\VT=\la \T_k, \P_k : k<m\ra$ is a correctly guided finite stack on $\P$ with last model $\Q$. We say $\vec{b}=\la b_k: k< m\ra$ witness $s$-iterability for $\VT=\la \T_k, \P_k : k< m\ra$ if 2 above is satisfied. We may also say that $\vec{b}$ is an $s$-iterability branch for $\VT$. We then let
 \begin{displaymath}
   \pi_{\VT, \vec{b}, k} = \left\{
     \begin{array}{lr}
      i^{\T_k} :&  \T_k \ has \ a \ successor \ length \\
      i_{b_k}^{\T_k} :& \T_k \ is \ maximal
     \end{array}
   \right.
\end{displaymath}
and $\pi_{\VT, \vec{b}}=\pi_{\VT, \vec{b}, m-1}\circ \pi_{\VT, \vec{b}, m-2}\circ\cdot \cdot\cdot \pi_{\VT, \vec{b}, 0}$.


Suppose now that $\vec{b}$ and $\vec{c}$ are two $s$-iterability branches for $\VT$. Then using \rlem{partial agreement}, it is easy to see that  $\pi_{\VT, \vec{b}}\rest H_{s}^\P=\pi_{\VT, \vec{c}}\rest H_s^\P$. Lets record this as a lemma.

\begin{lemma}[Uniqueness of $s$-iterability embeddings] \label{uniqueness of s-iterability embeddings} Suppose $\P$ is $n$-suitable, $s$ is a finite sequence of ordinals and $\VT$ is a finite correctly guided stack on $\P$. Suppose $\vec{b}$ and $\vec{c}$ are two $s$-iterability branches for $\VT$. Then
\begin{center}
$\pi_{\VT, \vec{b}}\rest H_{s}^\P=\pi_{\VT, \vec{c}}\rest H_s^\P$.
\end{center}
Moreover, if $\VT$ consists of just one normal tree $\T$, $\Q$ is the last model of $\T$ and $b$ and $c$ witness $s$-iterability for $\T$ then if $\xi\in b$ is the least such that $\cp(E^\T_\xi)>\gg_s^\Q$ then $b\cap \xi =c\cap \xi$. \end{lemma}

If $\P$ is $s$-iterable and $\T$ is a normal correctly guided tree then we let $b^\T_s=\cap \{b : b$ witnesses the $s$-iterability of $\P$ for $\T\}$. Here is how $s$-iterability is connected to iterability. Suppose $\P$ is $n$-suitable. We say $\P$ has a \textit{correct $\omega_1$-iteration strategy} if it has an $\omega_1$-iteration strategy $\Sigma$ such that whenever $\T$ is a correctly guided tree of limit length and $b=\Sigma(\T)$ then $\T^\frown \M^\T_b$ is correctly guided.

\begin{lemma}
Suppose $\P$ is $n$-suitable and for every $m$, $\P$ is $s_m$-iterable. Then $\P$ has a correct iteration strategy.
\end{lemma}
\begin{proof}
Let $\T$ be a correctly guided tree. If $\T$ is short then using $s_m$-iterability there must be a branch $b$ of $\T$ such that $\Q(b, \T)$-exists and is iterable. In this case we define $\Sigma(\T)=b$. Suppose now $\T$ is maximal with last model $\Q$. Then for each $m$, let $b_m=b^\T_{s_m}$. Notice that $b_m\subseteq b_{m+1}$. Also, because $\sup_{m\in \omega}\gg_m^\Q=\d^\Q$, we have that if $b=\cup_{m\in \omega} b_m$ then $b$ is a cofinal branch. We claim that $\M^\T_b=\Q$. Let $\R=\M^\T_b$. For all we know $\R$ may not be well-founded. But notice that if $R_m=i^\T_b(H_m^\P)$ then there is $\pi_m: \R_m\rightarrow_{\Sigma_1} H_m^\Q$. This is because $i_b^\T\rest \gg_m^\P=\pi_{\T, b_m}\rest \gg^\P_m$ where $b_m$ is any cofinal well-founded branch witnessing $s$-iterability of $\P$ for $\T$. It then follows that if $\pi=\cup_{m\in \omega} \pi_m$ then $\pi:\cup_{m\in \omega} \R_m \rightarrow \Q$ and because $\cup_{m\in \omega}\R_m=\R$, we have that $\R$ is well-founded. Because for each $i$ and $m$, $T^\Q_{s_m, i}\in ran(\pi)$, using the proof of \rlem{ggs sup up}, we get that $\R$ is $n$-suitable and hence, $\R=\Q$ and $\pi=id$. In this case, then, we define $\Sigma(\T)=b$. It follows from our construction that $\Sigma$ is a correct iteration strategy.
\end{proof}

Notice that, if $\P$ is $s$-iterable, $\VT$ is a correctly guided finite stack on $\P$, and $\vec{b}$ witnesses $s$-iterability of $\P$ for $\VT$, then even though $\pi_{\VT, \vec{b}}\rest H_{s}^\P$ is independent of $\vec{b}$ it may very well depend on $\VT$. This observation motivates the following definition.

\begin{definition}[Strong $s$-iterability]\label{strong s-itearbility} Suppose $\P$ is $n$-suitable and $s$ is a finite sequence of ordinals. Then $\P$ is strongly $s$-iterable if $\P$ is $s$-iterable and whenever $\VT=\la \T_j , \P_j : j< u\ra$ and $\VU=\la \U_j, \Q_j : j < v\ra$ are two correctly guided finite stacks on $\P$ with common last model $\Q$, $\vec{b}$ witnesses $s$-iterability for $\VT$ and $\vec{c}$ witnesses $s$-iterability for $\VU$ then
\begin{center}
$\pi_{\VT, \vec{b}}\rest H_s^\P = \pi_{\VU, \vec{c}}\rest H_s^\P$.
\end{center}
\end{definition}

Are there $s$-iterable $\P$'s? Of course there must be, as otherwise we wouldn't define them, and here is an argument that shows it. Suppose not. Let $s=\la \a_k : k< l\ra$. Using the fact that there are no $s$-iterable $\P$'s, we can then get $\vec{B}=\la B_k : k<\omega\ra$ such that $B_k=\la \T^k_j, \P^k_j , \Q_k: j< m_k\ra$ and
\begin{enumerate}
\item $\P^0_0=\W_n$ and $\P^{k+1}_0=\Q_k$,
\item for every $k$, $\la \T^k_j, \P^k_j : j< m_k\ra$ is a correctly guided finite stack on $\P^k_0$ with last model $\Q_k$,
\item whenever $\la b^k_j: k<\omega \wedge j<k< m_k \ra$ is such that
\begin{enumerate}
\item for $j<m_k-1$,
 \begin{displaymath}
   b^k_j = \left\{
     \begin{array}{lr}
      \emptyset :&  \T^k_j \ has \ a \ successor \ length \\
      cofinal\ well-founded\ branch\\ such\ that\ \M^\T_{b^k_j}=\P^k_j :& \T^k_j \ is \ maximal
     \end{array}
   \right.
\end{displaymath}
\item if $\T^k_{m_k-1}$ has a successor length then $b^k_{m_k-1}=\emptyset$, if $\T^k_{m_k-1}$ is short then $b^k_{m_k-1}$ is the unique cofinal well-founded branch such that $\Q(b^k_{m_k-1}, \T^k_{m_k-1})$ exists and is iterable, and if $\T^k_{m_k-1}$ is maximal then $b^k_{m_k-1}$ is a cofinal well-founded branch,\footnote{Notice that $s$-iterability cannot fail because we cannot find correct branches for short trees as long as we start with $\P_0=\W_n$.}
\end{enumerate}
then letting $\vec{b}_k=\la b^k_j: j< m_k\ra$, for some $m$ and every $k$,
\begin{center}
$\pi_{\VT_k, \vec{b}_k}(T^{\P^k_0}_{s, m})\not = T^{\Q_k}_{s, m}$.
\end{center}

\end{enumerate}
Let then $\la b^k_j: k<\omega \wedge j<m_k \ra$ be the sequence of branches given by $\Sigma_{\W_n}$. Then clearly for every $k$, $\pi_{\VT_k, \vec{b}_k}$'s extend to
\begin{center}
     $\pi_k:\M_n(\P^k_0)\rightarrow \M_n(\Q_k)$.
    \end{center}
Let $\b>\max(s)$ be a uniform indiscernible and let $t=s^\frown \la\b\ra$. Suppose that for some $k$, $\pi_k(t)=t$. Notice that for every $m$, $T^{\P^k_0}_{s, m}$ can be defined by
\begin{center}
$(t, \phi)\in T^{\P^k_0}_{s, m} \iff \phi$ is $\Sigma_1$ and $\M_n(\P^k_0)|\b\models \phi[t, s]$.
\end{center}
Hence, because we are assuming $\pi_k(t)=t$, we get that $\pi_{\VT_k, \vec{b}_k}(T^{\P^k_0}_{s, m}) = T^{\Q_k}_{s, m}$.

Therefore, we must have that $t<_{lex} \pi_k(t)$. Let $\Q$ be the direct limit of $\la \M_n(\Q_k): k< \omega \ra$ under the maps $\sigma_{k, l}= \pi_l\circ \pi_{l-1}\circ\cdot\cdot\cdot \pi_k$ and let $\pi^{*}_k: \M_n(\Q_k) \rightarrow \Q$ be the embedding given by the direct limit construction. Now if $t_k=\pi^{*}_k(t)$, then $\la t_k : k<\omega\ra$ is a $\leq_{lex}$-decreasing sequence of finite sequences of ordinals. Because $\pi^*_k$'s are iteration embeddings according to $\Sigma_{\W_n}$, we get a contradiction. This completes the proof that for every $s$ there is an $s$-iterable $n$-suitable $\P$. 

\begin{lemma}\label{s-iterability lemma}
For every $s\in Ord^{<\omega}$ and $n\in \omega$ there is an $s$-iterable $n$-suitable $\P$. Moreover, for any $n$-suitable $\Q$ there is a normal correctly guided tree $\T$ with last model $\P$ such that $\P$ is $s$-iterable.
\end{lemma}
\begin{proof}
We have already shown that there is an $s$-iterable $n$-suitable $\P$. It is then the second clause that needs a proof. Fix a $n$-suitable $\Q$ and let $\P$ be $s$-iterable. Comparing $\P$ and $\Q$ produces our desired $\T$.
\end{proof}

Is there a strongly $s$-iterable $\P$? The proof we have just given shows that there is. Indeed, using the proof given above we have $\P$ which is $s$-iterable and is a $\Sigma_{\W_n}$-iterate of $\W_n$. Moreover, if $\Lambda=\Sigma_\P$ then the branches witnessing $s$-iterability can be taken to be those given by $\Lambda$. It then easily follows from the Dodd-Jensen property of $\Lambda$ that $\P$ is strongly $s$-iterable. 

\begin{lemma}[Strongly $s$-iterability lemma]\label{strongly s-iterability lemma}
For every $s$ there is a strongly $s$-iterable $\P$. Moreover, for any $n$-suitable $\Q$ there is normal correctly guided stack $\T$ with last model $\P$ such that $\P$ is strongly $s$-iterable.
\end{lemma}
\begin{proof}
We have already shown that there is a strongly $s$-iterable $\P$. It is then the second clause that needs a proof. Fix a $n$-suitable $\Q$ and let $\P$ be a strongly $s$-iterable. Comparing $\P$ and $\Q$ produces our desired $\T$.
\end{proof}
If $\P$ is strongly $s$-iterable and $\VT$ is a correctly guided finite stack on $\P$ with last model $\Q$ then we let
\begin{center}
$\pi_{\P, \Q, s}:H_s^\P\rightarrow H_s^\Q$
\end{center}
be the embedding given by any $\vec{b}$ which witnesses the $s$-iterability of $\VT$, i.e., fixing $\vec{b}$ which witnesses $s$-iterability for $\VT$,
\begin{center}
$\pi_{\P, \Q, s} =\pi_{\VT, \vec{b}}\rest H_s^\P$.
\end{center}
Clearly, $\pi_{\P, \Q, s}$ is independent of $\VT$ and $\vec{b}$.

Notice that $\W_n$ is strongly $s_m$-iterable for every $m$. Moreover, if $\VT$ is any correctly guided stack on $\W_n$ with last model $\Q$ then $\pi_{\W_n, \Q, s_m}$ agrees with the correct iteration embedding, i.e., if $i:\W_n\rightarrow \Q$ is the iteration embedding according to the canonical iteration strategy of $\W_n$ then
\begin{center}
$\pi_{\W_n, \Q, s_m}=i\rest H_m^{\W_n}$.
\end{center}
Moreover, since $\cup_{m<\omega}H_m^{\W_n}=\W_n$, we get that
\begin{center}
$\cup_{m<\omega}\pi_{\W_n, \Q, s_m} = i$.
\end{center}
This is how we will approximate $\Sigma$ inside $\M_n(x)$.

Next let
\begin{center}
$\mathcal{F}^+_n=\{ \P: \P\in I(\W_n, \Sigma_{\W_n})$ as witnessed by some finite stack $\}$.
\end{center}
We let $\leq^+_n$ be a prewellording of $\mathcal{F}^+_n$ given by $\P\leq^+_n\Q$ iff $\Q\in I(\P, \Sigma_\P)$ as witnessed by a finite stack. We then let $\M_{\infty, n}^+$  be the direct limit of $(\mathcal{F}^+_n, \leq^+_n)$ under the iteration maps $i_{\P, \Q}$. Notice that $\card{R^+_n}=\d^{\M_\infty^+}$. We let $\d^+_{ \infty, n}=\d^{\M_{\infty, n}^+}$.

We also let
\begin{center}
$\mathcal{I}_n=\{ (\P, s): \P$ is $n$-suitable, $s\subseteq Ord^{<\omega}$ and $\P$ is strongly $s$-iterable $\}$.
\end{center}
and
\begin{center}
$\mathcal{F}_n=\{ H^\P_s: (\P, s)\in \mathcal{I}_n\}$.
\end{center}
We define $\leq_n$ on $\mathcal{I}_n$ by: $(\P, s)\leq_n (\Q, t)$ iff $\Q$ is a correctly guided iterate of $\P$ and $s\subseteq t$. Is $\leq_n$ directed? The answer is of course yes and to see that fix $(\P, s), (\Q, t)\in \mathcal{I}_n$. Then we have $\R$ which is strongly $s\cup t$-iterable. Let $\S$ be the result of comparing $\P, \Q$ and $\R$. Then $(\S, s\cup t)\in \mathcal{I}_n$ and
\begin{center}
$(\P, s)\leq_n(\S, s\cup t)$ and $(\Q, t)\leq_n (\S, s\cup t)$.
\end{center}
We can then form the direct limit of $(\mathcal{F}_n, \leq_n)$ under the maps $\pi_{\P, \Q, s}$. We let $\M_{\infty, n}$ be this direct limit. It is clear that $\M_{\infty, n}^+$ is well-founded. However, it is not at all clear that $\M_{\infty, n}$ is well-founded. We show that not only $\M_{\infty, n}$ is well-founded but that it is also the same as $\M_{\infty, n}^+$.

Before we continue, we fix some notation. If $\P\in I(\W_n, \Sigma_{\W_n})$, then we let $i_{\P, \infty}:\P\rightarrow \M_{\infty, n}^+$ be the iteration map. For $(\P, s)\in \mathcal{I}_n$, we let $\pi_{\P, \infty, s}$ be the direct limit embedding acting on $H_s^\P$.

\begin{lemma}\label{wellfoundness} $\M_{\infty, n}=\M_{\infty, n}^+$.
\end{lemma}
\begin{proof}
To show the equality, we  define a map $\pi:\M_{\infty, n} \rightarrow_{\Sigma_1} \M_{\infty, n}^+$ and show that $\pi$ is the identity. Let $x\in \M_{\infty, n}$. Let $(\P, s_m)\in \mathcal{I}_n$ be such that for some $y\in H_m^\P$, $\pi_{\P, \infty, s_m}(y)=x$ and $\P$ is a normal correct iterate of $\W_n$. Then we let
\begin{center}
$\pi(x)=i_{\P, \infty}(x)$.
\end{center}

First we need to see that $\pi$ is independent of the choice of $\P$. Let then $(\P, s_p)\in \mathcal{I}_n$ and $(\R, s_q)\in \mathcal{I}_n$ be such that there are $y\in H_p^\P$ and $z\in H_q^\R$ such that $\pi_{\P, \infty, s_p}(y)=\pi_{\R, \infty, s_q}(z)=x$ and both $\P$ and $\R$ are normal iterates of $\W_n$. Let $\Q$ be the outcome of comparing $\P$ and $\R$. Notice that we must have that
\begin{center}
$\pi_{\P, \Q, s_p}(y)=\pi_{\R, \Q, s_q}(z)$.
\end{center}
It then follows that
\begin{center}
$i_{\Q, \infty}(\pi_{\P, \Q, s_p}(y))=i_{\Q, \infty}(\pi_{\R, \Q, s_q}(z))$.
\end{center}
and hence, $\pi$ is independent of the choice of $\P$. A similar argument shows that $\pi$ is a $\Sigma_1$-elementary and this much is enough to conclude that $\M_{\infty, n}$ is well-founded. But we can in fact show that $\pi=id$. For this, fix $x\in \M_{\infty, n}^+|\d^{\M_{\infty, n}^+}$. Let $Q$ be such that there is $y\in \Q$ such that $x=i_{\Q, \infty}(y)$. Let $s_m$ be such that $y\in H_{m}^{\Q}$. Then if $z=\pi_{\Q,\infty, s_m}(y)$ then $\pi(z)=x$. This shows that $\pi\rest \d^{\M_{\infty, n}}+1=id$.

Now fix $\P$ and let $T^\infty_{m, l}=i_{\P, \infty}(T^\P_{m, l})$. We clearly have that $T^\infty_{m, l}\in ran(\pi)$.
Let then $S_{m, l}\in \M_{\infty, n}$ be such that $\pi(S_{m,l})=\T_{m,l}^\infty$. Now, let $\N=\M_{\infty, n}$. Then for each $l$, $\cup_{m<\omega}S_{m, l}$ is a prescription for constructing a model with $n$ Woodin cardinals over $\N|(\d^{+l})^\N$. Moreover, if $K$ is this model then $K$ is the $\Sigma_1$-hull of ordinals $<(\d^{+l})^\N$ and $\omega$ indiscernibles. Because of $\pi$, it follows that $K=\M_n^\#(\N|(\d^{+l})^\N)$. This then inductively implies that for every $l$, $\N|(\d^{+l})^\N=\S|(\d^{+l})^\S$ where $\S=\M_{\infty, n}^+$. Hence, $\pi$ has to be the identity.
\end{proof}

Before moving on, notice that everything we have done in this section relativizes to arbitrary real $x$. For any real $x$, we can define $\mathcal{J}_{x, n}^+$, $\mathcal{I}_{x, n}$, $\mathcal{F}^+_{x, n}$, $\mathcal{F}_{x, n}$, $\leq_{x, n}^+$, $\leq_{x, n}$, $\M_{\infty, x, n}^+$,  and $\M_{\infty, x, n}$. We will then again have that $\M_{\infty, x, n}^+=\M_{\infty, x, n}$ and $\d^{\M_{\infty, x, n}}<\d^1_{n+3}$. We let $\d_{\infty, x, n}=\d^{\M_{\infty, x, n}}$ and also for $s\in Ord^{<\omega}$, we let
\begin{center}
$\gg_{\infty, s, x, n}=\sup(\pi_{\P, \infty, s}"\gg^\P_s)$
\end{center}
where $(\P, s)\in \mathcal{I}_{x, n}$. Clearly $\gg_{\infty, s, x, n}$ is independent of the choice of $\P$.  We also let
\begin{center}
$\mathcal{J}_{n, s, z}=\{ (\P, \a) : (\P, s)\in \mathcal{I}_{n, z} \wedge \a<\gg_s^\P\}$.
\end{center}
We let $R_{n, s, z}$ be a prewellordering of $\mathcal{J}_{n, s, z}$ given by $(\P, \a)R_{n, s, z}(\Q, \b)$ if $\Q$ is a correct iterate of $\P$ and $\pi_{\P, \Q, s}(\a)\leq \b$. We also let $\W_z=\M_{n+1}(z)|(\d^{+\omega})^{\M_{n+1}(z)}$ where $\d$ is the least Woodin of $\M_{n+1}(z)$. We let $\Sigma_z$ be the strategy of $\M_{n+1}(z)$ restricted to stacks on $\W_z$. We now move to internalizing the direct limit construction to $\M_n(x)$ where $x$ is any real coding $\W_n$.

\subsection{Internalizing the directed system}\label{internalizing the directed system}

Fix a real $x$ that codes $\W_n$ and let $\d$ be the least Woodin of $\M_n(x)$. We will work with this $x$ until the end of this subsection. Notice that because $\M_{n}(x)|\d$ is closed under $S_n$ operator, if $\T\in \M_n(x)|\d$ is a short tree on $\W_n$ then if $b$ is such that $\T^\frown \M^\T_b$ is correctly guided then in fact $b\in \M_n(x)|\d$. Thus, $\Sigma_{\W_n}\rest\{\T\in \M_n(x)|\d: \T$ is short$\}$.

 How about maximal trees? We claim that $\Sigma_{\W_n}\rest \{ \T\in \M_n(x)|\d: \T$ is maximal$\}$ is not in $\M_n(x)$. To see this, assume otherwise.  By a result of Neeman from \cite{Neeman02}, there is a normal iterate $\Q\in HC^{\M_n(x)}$ of $\W_n$ via a tree of length $\omega$ such that there is some $\Q$-generic $g\subseteq Coll(\omega, \d^\Q)$ such that $g\in \M_n(x)$ and $x\in \Q[g]$. But this is a contradiction as $\Q$ is essentially a real in $\M_n(x)$ while $\mathbb{R}^{\M_n(x)}=S_n(x)\subseteq \Q[g]$. 
 
 Nevertheless, in the case of $n=0$, Woodin used $s$-iterability to track the iteration strategy of $\W_n$ inside $\M_n(x)$. We do that here for an arbitrary $n$. For the purpose of keeping the notation simple, while working in this subsection we let $\M=\M_n(x)$ and $\d$ be the least Woodin of $\M$. 

Notice that the notions such as suitable, short tree, maximal tree, correctly guided finite stack and etc are all definable over $\M$. This is because all these notions refer to the $S_n$ operator and $\M|\d$ is closed under the $S_n$ operator.  For instance, we have that $\Q\in \M|\d$,  $\Q$ is suitable iff $\M\models ``\Q$ is suitable". Notice, however, that $s$-iterability presents a difficulty as it is not immediately clear how to say ``a suitable $\P$ is $s$-iterable" inside $\M$. When $n=0$ and $s=\la a_j : j< l\ra$, one can just make do with \rdef{s-iterability}. This is because the ``guiding sets", $T^\P_{s, i}$, can be identified inside $L[x]$. In general, this doesn't seem to work because we need to correctly identify $T^\P_{s, i}$. If $\b>\max(s)$ is a uniform indiscernible then to identify $T^\P_{s, i}$ inside $\M$, it is enough to identify $\M_n(\P)|\b$ inside $\M_n(x)$. This is because
\begin{center}
$(t, \phi)\in T^\P_{s, m} \iff \phi$ is $\Sigma_1$ and $\M_n(\P)|\b\models \phi[t, s]$.
\end{center}
We then solve the problem by dropping to a smaller set of ``good" $\P$'s. This new set of good $\P$'s will nevertheless be dense in the old one. To start, we fix $\k<\d$ which is an inaccessible strong cutpoint cardinal of $\M$ such that $\M\models ``\k$ is a limit of strong cutpoint cardinals".

We let
\begin{center}
$\mathcal{G}_{\k}=\{ \P \in \M|\k: \P$ is suitable and $\M\models ``$ for some strong cutpoint $\eta$, $\d^\P=\eta^+$ and $\M|\eta$ is generic over $\P$ for $\d^\P$-generator version of the extender algebra at $\d^\P"\}$.
\end{center}
If $\P\in \mathcal{G}_{\k}$ then we let $\eta_\P$ be the ordinal witnessing that $\P\in \mathcal{G}_{\k}$. Recall the definition of $\S^\M(\N)$ (see \rdef{s(n)}).

\begin{lemma}\label{good points} Suppose $\P\in \M$ is suitable and such that for some strong cutpoint $\eta$ of $\M$, $\P|\d^\P\subseteq \M|(\eta^+)^\M$ and $\M|\eta$ is generic over $\P$ for the $\d^\P$-generator version of the extender algebra. Then $\P\in \mathcal{G}_\k$ and $\S^\M(\P)=\M_n(\P)$.
\end{lemma}
\begin{proof}
Notice that using the $\S$-constructions, we can rearrange $\M|(\eta^{+\omega})^{\M}$ as $\P[\M|\eta]$ (see \rprop{s-constructions prop}).
Hence, $\d^\P=(\eta^+)^\M$.
 But then $\S^{\M}(\P)[\M|\eta]=\M$. This means that $\S^{\M}(\P)$ is the hull of ordinals $<\d^\P$ and the class of indiscernibles. But this is exactly what $\M_n(\P)$ is: it is the unique proper class mouse over $\P$ with $n$ Woodin cardinals which is the hull of a club class of indiscernibles.
\end{proof}

Let $\P\in \mathcal{G}_{\k}$ and $s=\la \a_j : j< l\ra$. 

\begin{definition}
We then write $\M\models ``\P$ is $s$-iterable below $\k$" if whenever $\VT=\la \T_j, \P_j : j<k\ra\in \M|\kappa$ is a correctly guided finite stack on $\P$ with last model $\Q$ such that $\Q\in \mathcal{G}_{\k}$ and whenever $g\subseteq Coll(\omega, \card{\P\cup \Q})$ is $\M$-generic, there is $\vec{b}=\la b_j : j<k\ra\in \M[G]$ such that for every $m$,
\begin{center}
$\pi_{\VT, \vec{b}}(T^\P_{s, m})=T^\Q_{s, m}$
\end{center}
where $T^\P_{s, m}\subseteq [((\d^\P)^{+m})^\P]^{<\omega}\times \omega$ is defined by
\begin{center}
$(t, \phi)\in T^\P_{s, m} \iff \phi$ is $\Sigma_1$ and $\S^{\M}(\P)\models \phi[t, s]$.
\end{center}
and $T^\Q_{s, m}\subseteq [((\d^\Q)^{+m})^\Q]^{<\omega}\times \omega$ is defined by
\begin{center}
$(t, \phi)\in T^\Q_{s, m} \iff \phi$ is $\Sigma_1$ and $\S^{\M}(\Q)\models \phi[t, s]$.
\end{center}
\end{definition}

Notice that in the light of \rlem{good points}, the definition just given indeed coincides with \rdef{s-iterability} for as long as we stay inside $\mathcal{G}_{\k}$. $\M\models ``\P$ is strongly $s$-iterable below $\k$" is defined similarly. Also, notice that even though the requirement that the sequence $\vec{b}$ exists in the generic extension cannot be dropped, the embedding $\pi_{\P, \Q, s}$ is in $\M$ as it is unique and hence, it is in all generic extensions.

We then let
\begin{center}
$\mathcal{I}_{\k}=\{(\P, s) : \P\in \mathcal{G}_\k \wedge \M\models ``\P$ is strongly $s$-iterable below $\k"\}$.
\end{center}
and
\begin{center}
$\mathcal{F}_{\k}=\{ H^\P_s : (\P, s)\in \mathcal{I}_\k\}$.
\end{center}
Notice that the proof of \rlem{s-iterability lemma} can be used to show that for every $s$ there is $\P$ such that $(\P, s)\in \mathcal{I}_\k$. More formally, we have the following:

\begin{lemma}\label{internal s-iterability lemma} Suppose $\P\in \mathcal{G}_{\k}$ and $s$ is a finite sequence of ordinals. Then there is a normal correct iterate $\Q$ of $\P$ such that $(\Q, s)\in \mathcal{I}_{\k}$
\end{lemma}

Clearly, $\mathcal{F}_{\k}\in \M$. We then define $\leq_{\k}$ on $\mathcal{I}_{\k}$ by: $(\P, s)\leq_{\k} (\Q, t)$ iff $\Q$ is a correct iterate of $\P$ and $s\subseteq t$. It is not hard to see that $\leq_{\k}$ is directed.
\begin{lemma} $\leq_{\k}$ is directed
\end{lemma}
\begin{proof}
Fix $(\P, s), (\Q, t)\in \mathcal{I}_{\k}$. Then there is $(\R, s\cup t)\in \mathcal{I}_{\k}$. Working in $\M$, simultaneously compare $\P, \Q$ and $\R$ to get $\S^*\in \M|\kappa$. Let $\eta<\k$ be a strong cutpoint of $\M$ such that $\S^*\in \M|\eta$. Then iterate $\S^*$ to make $\M|\eta$-generic. This iteration produces $\S\in \M|\kappa$ such that $\d^\S=(\eta^+)^{\M}$. It then follows that  $(\S, s\cup t)\in \mathcal{I}_{\k}$ and $(\P, s), (\Q, t) \leq_\k (\S, s\cup t)$.
\end{proof}

Let then $\M_{\infty, \k}$ be the direct limit of $(\mathcal{F}_{\k}, \leq_{\k})$ under the embeddings $\pi_{\P, \Q, s}$. We first claim that $\M_{\infty, \k}$ is well-founded.
\begin{lemma}\label{internal wellfoundness} $\M_{\infty, \k}$ is well-founded.
\end{lemma}
\begin{proof} The proof is similar to the proof of \rlem{wellfoundness}. Let $\la \P_\a: \a<\k\ra\in \M$ be an enumeration of $\mathcal{G}_{\k}$. We construct a sequence $\la \Q^0_i, \T^0_i, \Q^1_i, \T^1_i : i<\omega\ra$ such that
\begin{enumerate}
\item $\Q^0_0=\W_n$ and $\T^l_i$ is a normal correctly guided tree on $\Q^l_i$ for $l= 0, 1$,
\item $\Q^1_{i}$ is the last model of $\T^0_i$ and $\Q^0_{i+1}$ is the last model of $\T^1_i$,
\item for every $\a<\k$, there is $i<\omega$ such that $\Q^0_i$ is a correct iterate of $\P_\a$,
\item $\Q^0_i\in \mathcal{G}_{\k}$.
\end{enumerate}
To construct such a sequence, we first fix $\la \eta_i : i<\omega\ra$ such that $\sup_{i<\omega}\eta_i=\k$. Suppose we have constructed $\la \Q^0_i, \T^0_i, \Q^1_i, \T^1_i : i\leq k\ra$. Let $\eta\in [\eta_i, \k)$ be a strong cutpoint of $\M$ such that $\la \Q^0_i, \T^0_i, \Q^1_i, \T^1_i : i\leq k\ra\in \M|\eta$. Thus, we actually have $\Q^0_{k+1}$. Then let $\Q^1_{k+1}$ be the result of simultaneously comparing all suitable $\P$'s such that $\P\in \M|\eta\cap \mathcal{G}_\k$. Notice that $\S$ is a normal correct iterate of every $\P \in \M|\eta\cap \mathcal{G}_\k$ including $\Q^0_{k+1}$. Let then $\T^0_{k+1}$ be the normal correctly guided tree on $\Q^0_{k+1}$ with last model $\Q^1_{k+1}$. The problem is that $\Q^1_{k+1}$ may not be in $\mathcal{G}_\k$. Let then $\nu\in (\eta, \k)$ be a strong cutpoint of $\M$ such that $\Q^1_{k+1}\in \M|\nu$. Iterate $\Q^1_{k+1}$ to make $\M|\nu$ generic for the extender algebra. Let then $\T^1_{k+1}$ be the resulting tree on $\Q^1_{k+1}$. Clearly $\Q^1_{k+1}\in \mathcal{G}_\k$ and the resulting sequence $\la \Q^0_i, \T^0_i, \Q^1_i , \T^1_i : i<\omega\ra$ is as desired.

Let then $\sigma_{j, k}=i_{\Q^0_j, \Q^0_k}$ and let $\Q$ be the direct limit of $\la \Q^0_j, \sigma_{j, k} : j<k<\omega\ra$. Then the proof of \rlem{wellfoundness} can be used to show that in fact $\Q=\M_{\infty, \k}$.
\end{proof}

Next we show that $\d^{\M_{\infty, \k}}=(\k^+)^{\M}$. For the purpose of keeping the notation nice, in this subsection we abuse the notation used in the previous subsection and whenever $(\P, s)\in \mathcal{I}_\k$, we write $\pi_{\P, \infty, s}$ for the direct limit embedding. Thus, $\pi_{\P, \infty, s}$ is an embedding that acts on $H_s^\P$ and embeds it into the corresponding structure in $\M_{\infty, \k}$. For each $s\in Ord^{<\omega}$, let $\gg_{\infty, s}=\sup (\pi_{\P, \infty, s}"\gg^\P_s)$ where $(\P, s)\in \mathcal{I}_\k$. Clearly, $\gg_{\infty, s}$ is independent of the choice of $\P$. Notice that $\d^{\M_{\infty, \k}}=\sup_{s\in Ord^{<\omega}} \gg_{\infty, s}=\sup_{m<\omega}\gg_{\infty, s_m}$. Our proof uses an idea that originated in Hjorth's work.

\begin{lemma}
$\d^{\M_{\infty, \k}}=(\k^+)^{\M}$.
\end{lemma}
\begin{proof}
First notice that for every $\a<\d^{\M_\infty, \k}$ there is in $\M$ a surjective map $f:\kappa\rightarrow \a$. To see this, first fix $s$ such that $\a<\gg_{\infty, s}$ and let $\la (\P_\b, \xi_\b): \b<\k\ra$ be an enumeration of the set $\{ (\P, \xi) : (\P, s)\in \mathcal{I}_\k \wedge \xi<\gg_s^\P\}$. Then let $f(\b)=\pi_{\P_\b, \infty, s}(\xi_\b)$. Clearly $\a\subseteq ran(f)$ and $f$ is onto. This observation shows that $\d^{\M_{\infty, \k}}\leq (\k^+)^{\M}$.

We therefore need to show that $\d^{\M_{\infty, \k}}\not < (\k^+)^{\M}$. Suppose then $\d^{\M_{\infty, \k}} < (\k^+)^{\M}$. We can then let $\leq^*\in \M$ be a well-ordering of $\k$ of length $\d^{\M_{\infty, \k}}$. Without loss of generality we assume $\k$ is least such that $\d^{\M_{\infty, \k}} <(\k^+)^{\M}$. It then follows that there is a formula $\phi$, a sequence $t\in [\k]^{<\omega}$ and an integer $m$ such that
\begin{center}
$\a\leq^*\b \iff \M\models \phi[t, s_m, \a, \b]$.
\end{center}
Now, fix $(\P, s_m)\in \mathcal{I}_\k$ such that $t\subseteq \l$ where $\l$ is the least measurable cardinal of $\P$. Let $\N=\M_n(\P)=\S^{\M}(\P)$. We have that $\M|\eta_{\P}$ is generic over $\P$ for the extender algebra of $\d^\P$. This means that $\N[\M|\eta_\P]$ can be reorganized as an $x$-mouse and in fact, $\N[\M|\eta_\P]=\M$. This then means that there are conditions $p$ which force that $\N[G]$ can be reorganized via $\S$-constructions as a mouse over a real and such that in $\N[G]$, $\d^{\M_{\infty, \k}}< (\k^+)^{\N[G]}$. Moreover, among those conditions there are also conditions that force that $\phi$ defines a well-ordering of $\k$ as above over $\N[G]$. Let then $D$ be the set of conditions $p$ of the extender algebra at $\d^\P$ such that $p$ forces that
\begin{enumerate}
\item $\N[G]$ can be reorganized as a premouse over a real,
\item $\N[G]\models ``\d^{\M_{\infty, \k}} < (\k^+)^{\N[G]}"$,
\item $\phi$ defines a well-ordering of $\k$ of length $(\d^{\M_{\infty, \k}})^{\N[G]}$.
\end{enumerate}

We let $\tau$ be the name of the prewellordering given by $\phi$. Consider now the set $B$ of pairs $(p, \a)$ such that $p\in D$, $\a<\l$ and for some $\xi$, in $p$ forces that the rank of $\a$ . Notice that whenever $(p, \a)\in B$ and $G$ is $\P$-generic such that $p\in G$, $\a$ has a rank in the well-ordering given by $\phi$ over $\N[G]$. We can then for each $\a<\l$ choose a maximal antichain of conditions $p$ such that $(p, \a)\in B$ and for some $\xi$, $p$ forces that $\a$ has rank $\xi$ in the well-ordering given by $\phi$. Let $\mathcal{A}_\a$ be such an antichain and let $\mathcal{A}=\{ (p, \a) : p\in \mathcal{A}_\a\}$. Notice that without loss of generality we can assume that $\mathcal{A}\in H_{m+1}^\P$. We then let $\mathcal{A}^\P=\mathcal{A}$.

For $(p, \a)\in \mathcal{A}$ let $\xi_{p, \a}$ be the rank of $\a$ as forced by $p$. Define $\leq^\P$ on $\mathcal{A}$ by $(p, \a)\leq^\P (q, \b)$ iff $\xi_{p, \a}\leq \xi_{q, \b}$. Notice that $\card{\leq^\P}$ is independent of the choice of $\mathcal{A}_\a$'s and $\card{\leq^\P}< \gg_{m+1}^\P$.

Define now a relation $R$ on the set $\{ (P, \xi) : \P\in \mathcal{G}_\k \wedge \xi<\gg_{m+1}^\P\}$ given by
\begin{center}
$R((\P, \xi), (\Q, \nu))$ if whenever $\R$ is such that $(\P, s_{m+1})\leq_\k (\R, s_{m+1})$ and $(\Q, s_{m+1})\leq_\k (\R, s_{m+1})$ then $i_{\P, \R, s_{m+1}}(\xi)\leq i_{\Q, \R, s_{m+1}}(\nu)$.
\end{center}
Clearly $R$ is well-founded and $\card{R}=\gg_{\infty, s_{m+1}}$.

Fix now an $\a<\k$. We say that $(\P, p)$ is a stable code for $\a$ if
\begin{enumerate}
\item $(\P, s_{m+1})\in \mathcal{I}_\k$,
\item $(p, \a)\in \mathcal{A}^\P$, $\xi^\P_{p, \a}=\card{\a}_{\leq^*}$, and whenever $\Q$ is a correct iterate of $\P$ such that $\Q\in \mathcal{G}_\k$,
\begin{center}
$\pi_{\P, \Q, s_{m+1}}(\card{\a}_{\leq^*})=\card{\a}_{\leq^*}$,
\end{center}
\item if $G\subseteq \mathbb{B}^\P$ is a generic object such that $x_G=\M|\eta_\P$ then $p\in G$.
\end{enumerate}
Notice that if $(\P, p)$ is a stable code for $\a$ then $\xi_{p, \a}^\P=\card{\a}_{\leq^*}$. This is because of condition 3, i.e., if $G\subseteq \mathbb{B}^\P$ is the generic so that $x_G=\M|\eta_\P$ then $\S(x_G)^{\M_n(\P)[G]}=\M$, $p\in G$ and $(\card{\a}_{\leq^*})^{\S(x_G)^{\M_n(\P)[G]}}=\card{\a}_{\leq^*}$.

We claim that for every $\a$ there is a stable code for $\a$. Let $\xi=\card{\a}_{\leq^*}$. To see this, suppose not. Let then $\P$ be such that $(\P, s_{m+1})\in \mathcal{I}_\k$, $\a<\l^\P$ and $\P$ is a correct iterate of $\W_n$. Then we can find $p\in \P$ such that $(p, \a)\in \mathcal{A}^\P$ and $(\P, p)$ satisfies 1 and 3 above. If it satisfies 2 then we are done, and therefore, we assume that $(\P, p)$ doesn't satisfy 2. Let then $(\P_0, p_0)=(\P, p)$ and let $\P_1$ witness the failure of 2. Thus, we have that $\xi=\xi^\P_{p, \a}$ and $i_{\P_0, \P_1, s_{m+1}}(\xi)>\xi$. But notice that there is $p_1\in \P_1$ such that $(p_1, \a)\in \mathcal{A}^{\P_1}$ and $\xi^{\P_1}_{p_1, \a}=\xi$. We then must have that $(\P_1, p_1)$ doesn't satisfy condition 2 above and therefore, we get $(\P_2, p_2)$ such that $\P_2\in \mathcal{G}_\k$ is a correct iterate of $\P_1$, $\pi_{\P_1, \P_2, s_{m+1}}(\xi)>\xi$ and $\xi^{\P_2}_{p_2, \a}=\xi$. In this fashion, by successively applying the failure of 2, we get a sequence $\la \P_i: i<\omega\ra$ such that for every $i$, $\P_{i}$ is a correct iterate of $\P_{i-1}$, for each $i$, $\P_{i}$ is a correct iterate of $\W_n$ and for $i\geq 0$,
\begin{center}
 $\pi_{\P_i, \P_{i+1}, s_{m+1}}(\xi)>\xi$.
\end{center}
Let then $\Q$ be the direct limit of $\la \P_i, i_{\P_i, \P_j} : i<j<\omega\ra$ and let $\sigma_i:\P_i \rightarrow \Q$ be the iteration embedding. Then because $\pi_{\P_i, \P_{i+1}, s_{m+1}}$'s agree with $i_{\P_i, \P_j}$, letting $\nu_i=\sigma_i(\xi)$ we get that $\la \nu_i : i<\omega\ra$ is a decreasing sequence of ordinals, contradiction! Thus, there is indeed a stable code for $\a$.

Now, for each $\a<\k$ choose $(\P_\a, p_\a)$ such that $(\P_\a, p_\a)$ is a stable code for $\a$. Let $\nu_\a=\card{(p, \a)}_{\leq^{\P_\a}}<\gg^{\P_a}_{m+1}$. Then we claim that for any $\a, \b<\k$, if $\a\leq^* \b$ then $R((\P_\a, \nu_\a), (\P_\b, \nu_\b))$. Indeed, let $\Q\in \mathcal{G}_\k$ be a common correct iterate of $\P_\a$ and $\P_\b$. Let $\nu=i_{\P_\a, \Q, s_{m+1}}(\nu_\a)$ and let $\zeta=i_{\P_\b, \Q, s_{m+1}}(\nu_\b)$. Let $\xi_\a=\card{\a}_{\leq^*}$ and $\xi_\b=\card{\b}_{\leq^*}$. We have that $i_{\P_\a, \Q, s_{m+1}}(\a)=\a$, $i_{\P_\b, \Q, s_{m+1}}(\b)=\b$,  $i_{\P_\a, \Q, s_{m+1}}(\xi_\a)=\xi_\a$ and $i_{\P_\b, \Q, s_{m+1}}(\xi_\b)=\xi_\b$. Because $\xi_\a\leq \xi_\b$, we have that
\begin{center}
$\card{(\pi_{\P_\a, \Q, s_{m+1}}(p_\a), \a)}_{\leq^\Q}\leq \card{(\pi_{\P_\b, \Q, s_{m+1}}(p_\b), \b)}_{\leq^\Q}$.
\end{center}
Therefore, $\nu\leq \xi$.

This shows that $\a \rightarrow (\P_\a, p_\a)$ is an order preserving map of $\leq^*$ into $R$ and hence,
\begin{center}
 $\card{\leq^*}\leq \card{R}=\gg_{\infty, s_{m+1}}<\d^{\M_\infty, \k}$.
\end{center}
\end{proof}

We finish by remarking that the directed limit of $\M$ at $\k$ is invariant under small forcing. This means that if $\mathbb{P}\in \M|\k$ and $g\subseteq \mathbb{P}$ is $\M$-generic then one can, working inside $\M[g]$, construct a directed system, much like we did above, and show that the direct limit of this system is the same as $\M_{\infty, \k}$. This mainly follows from Woodin's generic comparison process. The idea has been explained in various places and because of this we will omit it. The idea is as follows. It is enough to show it for $g$'s that are generic for $Coll(\omega, \eta^+)$ where $\eta<\k$ is a strong cutpoint. One then fixes a strong cutpoint $\nu<\k$ and performs a simultaneous comparison of all suitable pairs in $\M[g]|\nu$. It is then shown that the tree on $\W_n$ is in fact in $\M$. This follows from the homogeneity of the forcing. Let then $\P$ be the last of this comparison. We then get that $\P\in \M$ and it dominates all the suitable mice in $\M[g]|\nu$. This then easily implies that the directed system of $\M[g]$ is dominated by the one in $\M$, and hence, the direct limit of both systems must be the same. For more on the details of the generic comparison we refer the reader to \cite{CMI}, \cite{MSC} (Section 3.9) and  \cite{PFA}.

\subsection{The full directed system.}

In this subsection, we will establish some lemmas that connect the directed system associated with $\M_\omega$ with the directed system associated with $\M_{2k+1}$. In particular, we will prove \rthm{woodins thing}, originally due to Woodin, which has been widely known yet has remained unpublished for many years. We do not know if the proof of \rthm{woodins thing} presented here is the same or similar to Woodin's original proof. Woodin's result gives a characterization of $\k^1_{2k+1}$ in terms of cardinals of $\H$. We remind our readers that we assume that $\M_\omega^\#$ exists. This assumption is made for aesthetic reasons. Readers  familiar with the general theory can reduce the hypothesis to just $AD^{L(\mathbb{R})}$.

 In what follows, we will use superscript $f$ to indicate that we are dealing with the full directed system, i.e., with the system associated with $\M_\omega^\#$. Notice that because of \rthm{steel's thing}, for $\eta<(\d^2_1)^{L(\mathbb{R})}$, the notation $\H^{L(\mathbb{R})}|\eta$ makes sense.

Besides the proof of \rthm{woodins thing}, we will also prove \rlem{full limit is bounded} which we will use later on. When we talk about $\H$, we mean $\H^{L(\mathbb{R})}$. From now on until the end of the next subsection we fix $k\in \omega$. We will often omit superscripts or subscripts that usually would involve $k$ in them. By a standard Skolem hull argument done in $\H_z$, It follows from \rthm{steel's thing}, that there are many $\H$-cardinals $\nu$ such that $\M_{2k}(\H_z|\nu)\models ``\nu$ is Woodin". For each real $z$ let $\nu_z$ be the least such $\nu$.

Recall $\mathcal{F}$ of \rsec{dimt}. Next want to isolate a subset of $\mathcal{F}$ such that the direct limit of this subset will converge to $\M_{2k}(\H|\nu_0)|(\nu_0^{+\omega})^{\M_{2k}(\H|\nu_0)}$. For each real $z$, let $\eta_z$ be the least cardinal of $\M_\omega(z)$ such that $\M_{2k}(\M_\omega|\eta_z)\models ``\eta_z$ is Woodin". Then let $\W_z^f=\M_{2k}(\M_\omega(z)|\eta_z)|(\eta_z^{+\omega})^{\M_{2k}(\M_\omega(z)|\eta_z)}$. We let $\Sigma^f_z$ be the fragment of the $(\omega_1, \omega_1)$-strategy of $\M_\omega(z)$ that acts on stacks which are based on $\W_z^f$.  Let
\begin{center}
$\mathcal{F}^{+,f}_z=\{ \P: \P \in I(\W^f_z, \Sigma^f_z)$ as witnessed by a finite stack $\}$.
\end{center}
Whenever $\P, \Q\in \mathcal{F}^{+,f}_z$ and $\Q\in I(\P, (\Sigma^f_z)_\P)$, we will let $i_{\P, \Q}^f:\P\rightarrow \Q$ be the iteration embedding. Notice that in this notation we are omitting $z$ from subscripts and superscripts as it is usually clear what $z$ is. We hope this doesn't cause a confusion.

We can then define $\leq^f_z$ on $\mathcal{F}^{+, f}_z$ by $\P\leq^f_z\Q$ iff $\Q\in I(\P, (\Sigma^f_z)_\P)$. We let $\M_{\infty, z}^{+, f}$ be the direct limit of $(\mathcal{F}^{+,f}_z, \leq^{+, f}_z)$ under the iteration maps $i_{\P, \Q}^f$. We also let $i_{\P, \infty}^f:\P \rightarrow\M^{+, f}_{\infty, z}$ be the iteration map. Then clearly $\nu_z=\d^{\M_{\infty, z}^{+,f}}$.

Next we show that just like $\W_z$, $\mathcal{F}_z^{+, f}$ and $\leq^{+, f}_z$ can be internalized to $\M_{2k}(x)$ where $x$ codes $\W^f_z$. We first make the following definition.

\begin{definition}\label{miserable drop} Suppose $\P$ is suitable and $\T$ is a normal tree on $\P$. We say $\T$ has a \emph{miserable drop} if there is $\a<lh(\T)$ and ordinal $\eta$ such that if
\begin{center}
$\M=\cup \{\N: \M_\a^\T|\eta \trianglelefteq \N\trianglelefteq\M_\a^\T$ and $\eta$ is a strong cutpoint of $\N\}$
\end{center}
then the rest of $\T$ is a normal tree on $\M$ above $\eta$.
\end{definition}


\begin{lemma}\label{no misearable drops} Suppose $\Q, \R\in \mathcal{F}^{+,f}_z$. Let $\T$ on $\Q$ and $\U$ on $\R$ be the trees constructed via the comparison process in which $II$ uses $(\Sigma^f_z)_\Q$ on the $\Q$-side and $II$ uses $(\Sigma^f_z)_\R$ on the $\R$ side. Then $\T$ and $\U$ have no miserable drops.
\end{lemma}
\begin{proof}\label{no miserable drops}
Suppose towards a contradiction, $\T$ has a miserable drop. Let $\Q^*$ be the last model of $\T$ and $\R^*$ be the last model of $\U$. Then $i^\T$ cannot exist and therefore, it follows from the comparison lemma that $\R^*\vartriangleleft \Q^*$. Let $\a<lh(\T)$ be the largest such that there is a miserable drop in $\M_\a^\T$. Let $\eta$ be such that if
\begin{center}
$\M=\cup \{\N: \M_\a^\T|\eta \trianglelefteq \N\trianglelefteq\M_\a^\T$ and $\N$ is a premouse over $\M_\a^\T|\eta\}$
\end{center}
then the rest of $\T$ is a tree on $\M$ above $\eta$. It then follows that $\eta\in \R^*$. Notice that $\eta$ is a strong cutpoint in $\R^*$ and by fullness of $\R^*$, $\M\trianglelefteq \R^*$. Because $\Q^*$ is an iterate of $\M$ above $\eta$, we cannot have that $\M\trianglelefteq \Q^*$, contradiction!
\end{proof}

Our next lemma shows that if $\P, \Q\in \mathcal{F}_z^{+,f}$, then their comparison involves $\Q$-structures that are below the $S_{2k}$-operator.

\begin{lemma}\label{bound on q-structures} Suppose $\P, \Q\in \mathcal{F}_z^{+,f}$. Let $\R$ be the result of their comparison and let $\T$ and $\U$ be the trees on $\P$ and $\Q$ respectively that come from the comparison process. Then for every limit $\a$ such that $\a+1\leq lh(\T)$, if $b$ is the branch of $\T\rest \a$ chosen in $\T$ and $\Q(b, \T\rest\a)$-exists then $\Q(b, \T\rest \a)\trianglelefteq \M_{2k}(\M(\T\rest \a))$.
\end{lemma}
\begin{proof}
The reason for this is that the only way to produce normal trees with $\Q$-structures that are beyond $S_{2k}$-operator is to do a miserable drop. To see that our claim is true, assume not, and fix $\a$ such that $\a+1\leq lh(\T)$ and if $b$ is the branch of $\T\rest \a$ chosen in $\T$ such that $\Q(b, \T\rest\a)$-exists then $\Q(b, \T\rest \a)\not \trianglelefteq \M_{2k}(\M(\T\rest \a))$. It then follows that $\M_{2k}(\M(\T\rest \a))\vartriangleleft \Q(b, \T\rest \a)$ and therefore, $\M_{2k}(\M(\T\rest \a))\models  ``\d(\T\rest \a)$ is Woodin". Notice that it follows from the comparison lemma and the minimality condition on $\P$ that $i^{\T\rest \a}_b$ exists (i.e., there are no drops along $b$). This means that $\a+1< lh(\T)$. But then $lh(E^\T_\a)>\d(\T\rest\a)$. Because $\R$ agrees with $\M^\T_\a$ up to $lh(E^\T_\a)$ and $\R\models ``lh(E^\T_\a)$ is a cardinal", $\d(\T\rest \a)$ is a cardinal in $\R$ and moreover, $\M_{2k}(\R|\d(\T\rest \a))\models ``\d(\T\rest \a)$ is Woodin". This means that $\d(\T\rest \a)=\d^\R$. 

Notice now that we must have that $\cp(E^\T_\a)\leq \d(\T\rest \a)$. To see this assume not. We then have that $\cp(E^\T_\a)>\d(\T\rest \a)$. But because $\d(\T\rest \a)=\d^\R$, we have that there must be a miserable drop in $\T$ at stage $\a+1$ (as we must start iterating above $\d(\T\rest \a)$). 

It now follows that $\M_\a^\T\models ``\cp(E^\T_\a)$ is a limit of cardinals $\eta$ such that $\M_{2k}(\M_\a^\T|\eta)\models ``\eta$ is Woodin". Because of the agreement between $\M_\a^\T$ and $\R$, we get that there is an $\R$-cardinal $\eta<\d^\R$ such that $\M_{2k}(\R|\eta)\models ``\eta$ is Woodin". This is a contradiction. 
\end{proof}

Using miserable drops, we can now define $s$-iterability for $\P\in \mathcal{F}^{+,f}_z$. First, given an iteration tree $\T$ on $\P$, we say $\T$ is \textit{correctly guided} if $\T$ doesn't have miserable drops and for every limit $\a<lh(\T)$, if $b$ is the branch of $\T\rest \a$ chosen by $\T$ and $\Q(b, \T\rest \a)$ exists then $\Q(b, \T\rest \a)\trianglelefteq \M_{2k}(\M(\T\rest \a))$. $\T$ is \textit{short} if there is a well-founded branch $b$ such that $\T^\frown \{\M^\T_b\}$ is correctly guided. $\T$ is \textit{maximal} if $\T$ is not short. One can then proceed and define $s$-iterability as in \rdef{s-iterability}: the only difference is that we require that the trees in the stack be without miserable drops. We define $T_{s, m}^\P$, $\gg_s^\P$ and $H_s^\P$ as before and we omit $z$ from superscripts and subscripts as that is really part of $\P$. Notice that
\begin{center}
$\sup_{m\in \omega}\gg_{s_m}^\P=\d^\P$.
\end{center}
For $\P, \Q\in \mathcal{F}^{+, f}_z$, we say $\Q$ is a \textit{correct iterate} of $\P$ if there is a correctly guided finite stack $\VT$ on $\P$ with last model $\Q$.

Suppose now $\P$ and $\Q$ are two correct iterates of $\W^f_z$. Then using the proof of \rlem{no miserable drops}, we can show that the comparison of $\P$ and $\Q$ can be entirely, except possibly the very last step, be carried out in $\M_{2k}(\P, \Q)$. That is, one can show that there are correctly guided trees $\T, \U\in \M_{2k}(\P, \Q)$ such that $\T$ is on $\P$, $\U$ is on $\Q$ and $\T$ and $\U$ have a common last model.

Using this observation and the results of \rsec{internalizing the directed system} one can internalize the directed system associated to $\W_z^f$. More precisely, suppose $x$ is a real coding $\W_z^f$ and $\k$ is an inaccessible strong cutpoint of $\M_{2k}(x)$ such that $\k$ is below the first Woodin of $\M_{2k}(x)$ and $\k$ is a limit of strong cutpoints, then one can form the direct limit of all correct iterates of $\W^f_z$ that are in $\M_{2k}(x)$. Notice that in \rsec{internalizing the directed system}, our internalization process didn't use $\W_z$ as a parameter in the definition. Here too we could make do without $\W_z^f$ but we don't it. Before we move on, let us then lay down the notation that is slowly evolving and becoming rather cumbersome.

\begin{enumerate}
\item We let $\mathcal{F}^{+,f}_z=\{ \P: \P$ is a correct iterate of $\W^f_z\}$, $\mathcal{J}^{+,f}_z=\{ (\P, \a) : \P\in \mathcal{F}^{+,f}_z \wedge \a<\d^\P\}$, and $\R_z^{+,f}$ is the prewellordering defined on $\mathcal{J}^{+,f}_z$ by:
    \begin{center}
    $(\P, \a)R^{+,f}_z (\Q, \b)$ iff $\Q$ is a correct iterate of $\P$ and $i^f_{\P, \Q}(\a)\leq \b$.
    \end{center}
We let $\leq^{+, f}_z$ be the prewellordering of $\mathcal{F}^{+, f}_z$ given by:
\begin{center}
    $\P\leq^{+,f}_z \Q$ iff $\Q$ is a correct iterate of $\P$.
    \end{center}

\item We let $\mathcal{I}^f_z=\{ (\P, s): \P\in \mathcal{F}^f_z \wedge s\in Ord^{<\omega}\wedge \P$ is strongly $s$-iterable $\}$, $\mathcal{F}^{f}_z=\{ H_s^\P : (\P, s)\in \mathcal{I}^f_z\}$ and $\mathcal{J}^{f}_{z, s}=\{ (\P, \a) : \P\in \mathcal{F}^{+,f}_z \wedge \a<\gg^\P_s\}$. We let $R^{f}_z$ be the prewellordering of $\mathcal{J}^{f}_z$ given by:
    \begin{center}
    $(\P, \a)R^{f}_{z, s}(\Q, \b)$ iff $\Q$ is a correct iterate of $\P$ and $\pi_{\P, \Q, s}(\a)\leq \b$.
    \end{center}
    We let $\leq^f_z$ be the prewellordering of $\mathcal{I}^f_z$ given by:
    \begin{center}
    $(\P, s)\leq^f_z(\Q, t)$ iff $\Q$ is a correct iterate of $\P$ and $s\subseteq t$. We have that $\leq^f_z$ is directed.
    \end{center}

\item Given $\P$ and $s\in Ord^{<\omega}$ such that $(\P,s)\in \mathcal{I}^f_z$, if $\Q$ is a correct iterate of $\P$ then we let $\pi^f_{\P, \Q, s}:H_s^\P\rightarrow H_s^\Q$ be the $s$-iterability embedding. $z$ will be clear from the context and hence, we omit it. Recall that we let $\pi_{\P, \Q, s}:H_s^\P\rightarrow H_s^\Q$ be the $s$-iterability embedding where $\P, \Q$ are suitable $\P$ is $s$-iterable and $\Q$ is a correct iterate of $\P$.

\item We let $\M_{\infty, z}^f$ be the direct limit of $(\mathcal{F}^f_z, \leq^f_z)$ under the maps $\pi^f_{\P, \Q, s}$ and $\M^{+, f}_{\infty, z}$ be the direct limit of $(\mathcal{F}^{+, f}_z, \leq^{+, f}_z)$ under the iteration maps $i^f_{\P, \Q}$. By the proof of \rlem{wellfoundness}, $\M^{+, f}_{\infty, z}=\M_{\infty, z}^f$.
\item We let $\pi_{\P, \infty, s}^f:H_s^\P\rightarrow_{\Sigma_1} \M_{\infty, z}^f$ and $\pi_{\P, \infty, s}:H_s^\P\rightarrow_{\Sigma_1} \M_{\infty, z}$ be the corresponding iteration embeddings.
\item Recall that $\d_{\infty, z}=\d^{\M_{\infty, z}}$. We also let $\d^f_{\infty, z}=\d^{\M_{\infty, z}^f}$. Thus, $\d^f_{\infty, z}=\nu_z$ (this follows from \rthm{steel's thing}).
\item We let $\gg^f_{\infty, s, z}=\sup \pi^f_{\P, \infty, s}"\gg_s^\P$ for some $\P$ such that $(\P, s)\in \mathcal{I}^f_z$. Recall that $\gg_{\infty, s, z}=\sup \pi_{\P, \infty, s}"\gg_s^\P$ for some $\P$ such that $(\P, s)\in \mathcal{I}_z$.

\item We let $\M_{\infty, \k, z, x}^f$ be the direct limit of $\W^f_z$ constructed inside $\M_{2k}(x)$ at $\k$. Here $x$ codes $\W^f_z$ and $\k$ is an inaccessible strong cutpoint of $\M_{2k}(x)$ which is less than the first Woodin of $\M_{2k}(x)$ and is a limit of strong cutpoints of $\M_{2k}(x)$.
\item We let $\M_{\infty, \k, z, x}$ be the direct limit of $\W_{2k+1,z}$ constructed inside $\M_{2k}(x)$. Here $x$ codes $\W_{2k+1, z}$ and $\k$ is an inaccessible strong cutpoint of $\M_{2k}(x)$ which is less than the first Woodin of $\M_{2k}(x)$ and is a limit of strong cutpoints of $\M_{2k}(x)$.
\item We let $\M_{\infty, z, x}^f=\M_{\infty, \k, z, x}^f$ and $\M_{\infty, z, x}=\M_{\infty, \k, z, x}$  where $\k$ is the least inaccessible of $\M_{2k}(x)$.

\item $\pi^f_{\P, \infty, s, x}:\P\rightarrow \M_{\infty, z, x}^f$ and  $\pi_{\P, \infty, s, x}:\P\rightarrow \M_{\infty,z, x}$ be the corresponding iteration embeddings.
\item If $a$ is a countable transitive set such that $\W_z\in a$ or $\W^f_z\in a$ then we let $\M_{\infty, \k, z, a}^f$, $\M_{\infty, \k, z, a}$, $\M_{\infty, z, a}^f$, $\M_{\infty, z, a}$, $\pi^f_{\P, \infty, s, a}$, and $\pi_{\P, \infty, s, a}$ be the corresponding objects.

\end{enumerate}


 Our first lemma is that $R_{z}$ dominates $R_{z}^f$.

\begin{lemma}\label{full limit is bounded} For every $z$ if $w$ is a real coding $\W_z^f$ then for every $m$, $\card{R^f_{z, s_m}}\leq \card{R_{w, s_m}}$.
\end{lemma}
\begin{proof}
Fix $z$, $w$ and $m$ as in the hypothesis.
 We now construct an order preserving embedding $f:\card{R^f_{z, s_m}}\rightarrow \card{R_{z, s_m}}$.

Suppose $\P$ is such that $(\P, s_m)\in\mathcal{I}_{w}$. By iterating if necessary, we get that there are conditions in the extender algebra of $\P$ that force that the generic object is a pair $(\Q, \a)\in \mathcal{J}^f_{z, s_m}$. The formula expressing this has $\W_z^f$ as a parameter and essentially says that $\Q$ is a correct iterate of $\W_z^f$ and $\a<\gg_m^\Q$. Because if $G\subseteq Coll(\omega, \d^\P)$ is $\M_{2k}(\P)$-generic and $x_g\in \M_{2k}(\P)[g]$ is the real coding $\P|\d^\P$ then we can form $\M_{\infty, z, x_g}^f$ \footnote{Notice that one can show via $\S$-constructions that $\M_{2k}(\P)[g]=\M_{2k}(x)$.}, there are conditions $p$ in the extender algebra of $\P$ that decide values for $\pi^f_{\dot{\Q}, \infty, \check{s}_m, x_g}(\check{\a})$ where $(\dot{\Q}, \a)$ is the generic object containing $p$. Notice that the value of $\pi^f_{\dot{\Q}, \infty, \check{s}_m, x_g}(\check{\a})$ is independent of $g$. We then let $\mathcal{A}^\P$ be a maximal antichain of conditions $p$ such that
\begin{enumerate}
\item $p$ forces that the generic object is a pair $(\Q, \a)\in \mathcal{I}^f_{z, s_m}$,
\item for some $\b$, $\M_{2k}(\P)\models ``p\forces_{Coll(\omega, \d^\P)} \pi^f_{\dot{\Q}, \infty, \check{s}_m, x_g}(\check{\a})=\check{\b}"$.
\end{enumerate}
Notice that we can assume that $\mathcal{A}^\P\in H_{s_m}^\P$. For each $p\in \mathcal{A}^\P$ let $\b_p$ be the witness for 2. We can then define $\leq^\P$ on $\mathcal{A}^\P$ by: $p\leq^\P q \iff \b_p\leq \b_q$. Notice that $\card{\leq^\P}<\gg_{s_m}^\P$. We have that $p\leq^\P q$ iff $\M_{2k}(\P)\models (p, q) \forces ``$ if $\dot{G}=((\dot{\Q}, \check{\a}), (\dot{R}, \check{\b}))$ then $(\dot{Q}, \check{\a})R^f_{\check{z}, \check{s}_m} (\dot{\R}, \check{\b})"$.

Fix now $(\Q, \a)\in \mathcal{I}^f_{z, s_m}$. We say $(\P, p)$ is $(\Q, \a)$-stable if
\begin{enumerate}
\item $(\Q, \a)$ is generic for the extender algebra of $\P$ and $p\in G$ where $G\subseteq \mathbb{B}^\P$ is the generic object such that $x_G=(\Q, \a)$,
\item $p\in \mathcal{A}^\P$ and $\b_p=\pi^f_{\Q, \infty, s_m, \P[\Q]}(\a)$,
\item whenever $(\R, q)$ is such that $\R$ is a correct iterate of $\P$ such that $(\Q, \a)$ is generic over $\R$ for the extender algebra at $\d^\R$ and letting $G\subseteq \mathbb{B}^\R$ be the generic such that $x_G=(\Q, \a)$,  $q\in \mathcal{A}^\R\cap G$,
    \begin{center}
    $\b_q=\pi_{\P, \R, s_m}(\b_p)$.
    \end{center}
    Thus, $q=_{\leq^\R}\pi_{\P, \R, s_m}(p)$.
\end{enumerate}

We claim that for every $(\Q, \a)\in \mathcal{I}^f_{z, s_m}$ there is a $(\Q, \a)$-stable $(\P, p)$. To see this assume not and fix $(\Q, \a)\in \mathcal{I}^f_{z, s_m}$ such that there is no $(\Q, \a)$-stable pair $(\P, p)$. Let $\P_0$ be such that $(\Q, \a)$ is generic for the extender algebra of $\P_0$. Letting $G\subseteq \mathbb{B}^\P$ be the generic object such that $x_G=(\Q, \a)$, we have a unique condition $p_0\in \mathcal{A}^\P\cap G$. Because $(\P_0, p_0)$ isn't $(\Q, \a)$-stable, there is $\P_1$ which is a correct iterate of $\P_0$ and is such that $(\Q, \a)$ is generic over $\P_1$ for the extender algebra at $\d^{\P_1}$ and if $p_1\in \mathcal{A}^{\P_1}\cap H$ where $H\subseteq \mathbb{B}^{\P_1}$ is the $\P_1$-generic such that $x_H=(\Q, \a)$ then
\begin{center}
$\b_{p_1}\not=\pi_{\P, \R, s_m}(\b_{p_0})$
\end{center}
Let
\begin{center}
$i=_{def}i_{\M_{2k}(\P_0), \M_{2k}(\P_1)}\rest \M^f_{\infty, z, \P_0}:\M^f_{\infty, z, \P_0} \rightarrow \M_{\infty, z, \P_1}^f$.
\end{center}
Then by Dodd-Jensen we have that
\begin{center}
$i(\pi^{f}_{\Q,\infty, s_m, \P_0}(\a))\geq \pi^f_{\Q,\infty, s_m, \P_1}(\a)$,
\end{center}
implying that
\begin{center}
$i(\b_{p_0})\geq \b_{p_1}$.
\end{center}
But because $i(\b_{p_0})=\pi_{\P, \R, s_m}(\b_{p_0})$ and $\b_{p_1}\not= \pi_{\P, \R, s_m}(\b_{p_0})$, we get that
\begin{center}
$\b_{p_1}<i(\b_{p_0})$.
\end{center}
Continuing this construction we get $\la \P_k, p_k: k<\omega\ra$ such that $\P_0$ is a correct iterate of $\W_w$, $\P_{k+1}$ is a correct iterate of $\P_k$ and $\b_{p_{k+1}}< i_{\P_k, \P_{k+1}}(\b_{p_k})$. Let then $\P$ be the direct limit of $\P_k$'s under the embeddings $i_{\P_k, \P_{k+1}}$ and let $\sigma_k:\P_k\rightarrow \P$ be the direct limit embedding. Then letting $\xi_k=\sigma_k(\b_{p_k})$, we get that $\la \xi_k : k\in \omega\ra$ is a descending sequence of ordinals, contradiction.

For each $(\Q, \a)\in \mathcal{I}^f_{z, s_m}$ let $A_{\Q, \a}=\{ (\P, p): (\P, p)$ is $(\Q, \a)$-stable $\}$. Let $B_{\Q, \a}=\{ (\P, \xi ): \exists p ( (\P, p)\in A_{\Q, \a} \wedge \card{p}_{\mathcal{A}^\P}=\xi)\}$. Then notice that if $(\P_i, \xi_i)\in B_{\Q_i, \a_i}$ for $i=0,1$ then
\begin{center}
$(\Q_0, \a_0)R^f_{z, s_m} (\Q_1, \a_1)\iff (\P_0, \xi_0)R_{w, s_m} (\P_1, \xi_1)$
\end{center}
To see this, let $\P$ be a common correct iterate of $\P_0$ and $\P_1$ such that $(\Q_0, \a_0)$ and $(\Q_1, \a_1)$ are generic for the extender algebra of $\P$. Then let $G_i\subseteq \mathbb{B}^{\P}$ be the $\P$-generic such that $x_{G_i}=(\Q_i, \a_i)$ ($i=0, 1$). Let $p_i\in \mathcal{A}^{\P}\cap G_i$. Suppose now $i_{\P_0, \P}(p_0)\leq^\P i_{\P_1, \P}(p_1)$. Because of stability we have that
\begin{center}
$i_{\P_k, \P}(\b_{p_k})=\pi_{\Q_k, \infty, s_m, \P}(\a_\k)\ \ \ k=0,1$.
\end{center}
Because $i_{\P_k, \P}(\b_{p_k})=\b_{i_{\P_k, \P}(p_k)}$ ($k=0,1$) and $i_{\P_0, \P}(p_0)\leq^\P i_{\P_1, \P}(p_1)$, we get that
\begin{center}
$\pi_{\Q_0, \infty, s_m, \P}(\a_0)\leq \pi_{\Q_1, \infty, s_m, \P}(\a_1)$.
\end{center}
This then implies that
\begin{center}
$\M_{2k}(\P)[(\Q_0, \a_0), (\Q_1, \a_1)]\models ``(\Q_0, \a_0)R_{z, s_m}^f (\Q_1, \a_1)$".
\end{center}
Hence, $(\Q_0, \a_0)R_{z, s_m}^f (\Q_1, \a_1)$. The other direction is similar.

Let then $f: \card{R^f_{z, s_m}}\rightarrow \card{R_{w, s_m}}$ be given by $f(\nu)=\eta$ if whenever $(\Q, \a)\in \mathcal{I}^f_{z, s_m}$ is such that $\card{(\Q, \a)}_{R^f_{z, s_m}}=\nu$ then for any $(\P, \b)\in B_{\Q, \a}$, $\card{(\P, \b)}_{R_{w, s_m}}=\eta$. The proof just used can be easily modified to show that $f$ is order preserving and hence, $\card{R^f_{z, s_m}}\leq \card{R_{w, s_m}}$.
\end{proof}

The proof of \rlem{full limit is bounded} can be used to prove the following.

\begin{corollary}\label{corollary to boundness} For any $m\in \omega$ and $z, w\in \mathbb{R}$, if $z\leq_T w$ then $\card{R_{z, s_m}}\leq \card{R_{w, s_m}}$ and $\card{R}_{z, s_m}\leq \card{R^f_{z, s_m}}$.
\end{corollary}

Next, we prove Woodin's result. The proof presented here is due to the author. We are grateful to Woodin for letting us state and proof this very useful lemma.

\begin{theorem}[Woodin]\label{woodins thing} Assume $AD+V=L(\mathbb{R})$. For $k\in \omega$, $\k^1_{2k+3}$ is the least cardinal $\d$ of $\H$ such that \begin{center}
$\M_{2k}(\H|\d) \models ``\d$ is Woodin".
\end{center}
\end{theorem}
\begin{proof}
Again, we prove the theorem from the assumption that $\M_\omega^{\#}$ exists. However, readers familiar with the general theory surrounding this topic can reduce the hypothesis to just $AD^{L(\mathbb{R})}$. It easily follows from \rlem{no miserable drops} and the remarks following it that for each $z\in \mathbb{R}$, $\card{R^{+,f}_z} < \d^1_{2k+3}$.  To finish the proof of \rthm{woodins thing}, we need then to show that for all reals $z$, $\d^f_{\infty, z}\leq \k^1_{2k+3}$ and that $\d^f_{\infty, z}\geq \k^1_{2k+3}$. We start with the first.

Suppose that for some $z$, $\d^f_{\infty, z} > \d^1_{2k+3}$. Let $U\subseteq \mathbb{R}$ be the set
\begin{center}
$\{ (x, y) : y$ codes $\Pi^1_{2k+2}$-iterable premouse $\M$ over $x$ such that $\M$ has $2k+1$ Woodins, proper initial segments of $\M$ are 2k+1-small and $\M$ has a last extender$\}$.
\end{center}
Then $U$ is $\Pi^1_{2k+2}$ and we can let $T\subseteq \omega^{<\omega}\times \omega^{<\omega}\times (\k^1_{2k+3})^{<\omega}$ be a tree such that $p[T]=U$. It follows by \rthm{steel's thing} that for every $w$
\begin{center}
 $\M^f_{\infty, w}|\d^f_{\infty, w}=\H|\d^f_{\infty, w}$.
\end{center}
Therefore, there is $w$ which codes $\W^f_z$ and is such that $T \in \M^f_{\infty, w}|\eta$ for some $\eta<\d^f_{\infty, w}$ (because we are assuming that $\d^f_{\infty, z}>\k^1_{2k+3}$ and by \rlem{full limit is bounded}, we have that $\d^f_{\infty, z}\leq\d^f_{\infty, w}$). Let then $\P\in \mathcal{F}^f_w$ be such that there is $S \in \P|\d^\P$ such that $i^f_{\P, \infty}(S)=T$. We can fix $l$ such that $S \in H_l^\P$. Let $u$ be a real coding $(\W^f_w, \P)$. Let $S^*=\pi^f_{\P, \infty, s_l, \W_u}(S)$. We claim that $\M_{2k+1}(u)\models ``p[S^*_u]\not =\emptyset"$.

To see that $\M_{2k+1}(u)\models ``p[S^*_u]\not =\emptyset"$, fix a correct iterate $\R$ of $\P$ such that for some $y$ there is $h\in (\gg_l^\R)^\omega$ such that if $g=\pi^f_{\R, \infty, s_l}"h$ then $(u, y, g)\in [T]$. Notice that $\M_{2k+1}(u)=\M_{2k}(\W_u)$. Iterate $\W_u$ to make $(\R, y)$ generic. Let $\Q$ be this iterate. Let $\bar{g}=\pi^f_{\R, \infty, s, \Q[\R, y]}" h$. Then for every $k$, we must have that
\begin{center}
$(y\rest k, \bar{g}\rest k)\in (\pi^f_{\W_w^f, \infty, s_l, \Q}(S))_u=(\pi^f_{\R, \infty, s_l, \Q[\R, y]}(S))_u$.
\end{center}
This means that $[(\pi^f_{\W_w^f, \infty, s_l, \Q}(S))_u]\not =\emptyset$. By absoluteness we have that
\begin{center}
  $\M_{2k}(\Q)\models [(\pi^f_{\W_w^f, \infty, s_l, \Q}(S))_u]\not =\emptyset$.
\end{center}
It then follows by elementarity that
\begin{center}
$\M_{2k+1}(u)\models ``p[S^*_u]\not =\emptyset"$.
\end{center}
It is, however, a well-known fact that there cannot be $y\in \M_{2k+1}(u)$ which codes a $\Pi^1_{2k+2}$-iterable premouse $\M$ over $u$ such that the proper initial segments of $\M$ are 2k+1-small and $\M$ has $2k+1$-Woodins and a last extender.\footnote{One way to see this is to use a result from \cite{PWOIM}. It is shown there that $x\in \Q_{2k+3}(u) \iff x$ is in every $\Pi^1_{2k+2}$-iterable premouse $\M$ such that the proper initial segments of $\M$ are 2k+1-small and $\M$ has $2k+1$ Woodins and a last extender. Thus, if there was such a premouse $\M\in \M_{2k+1}(u)$ then as $Q_{2k+3}(u)=\mathbb{R}^{\M_{2k+1}(u)}$, $\M\in \M$, contradiction!} This contradiction shows that $\d^f_{\infty, z}\leq \k^1_{2k+3}$.

To show that $\d^f_{\infty, z}\geq \k^1_{2k+3}$, it is enough to show that $\d_{\infty, 0} \geq \k^1_{2k+3}$. For this, we show that every $\Pi^1_{2k+2}$-set is $\d_{\infty, 0}$-Suslin. Let $\d=\d_{\infty, 0}$. To see that the universal $\Pi^1_{2k+2}$-set is $\d$-Suslin let $\Q=\M_{2k}^\#(\M_{\infty, 0}|\d)$. Notice that $\Q$ has size $\d$. Let $U$ be the universal $\Pi^1_{2k+2}$-set. Let $\phi$ be $\Pi^1_{2k+2}$ such that $x\in U\iff \phi(x)$. Let $T$ be the tree of attempts to construct a triple $( x, z, \pi )$ such that
\begin{enumerate}
\item $z$ codes a premouse $\M_z$,
\item $\pi: M_z\rightarrow \Q$,
\item $x$ is generic over $M_z$ for the extender algebra at the least Woodin of $\M_z$,
\item $M_z[x]\models \phi[x]$.
\end{enumerate}
Let then $S=\{(s, f) : s\in \omega^{<\omega}$, $f\in [\d]^{<\omega}$ and $f$ codes $f_0, f_1$ such that $(s, f_0, f_1)\in T\}$. Then, because $\M_{2k}(z)$ is $\Pi^1_{2k+2}(z)$-correct, it is not hard to see that $p[S]=U$. This then completes the proof that $\d^f_{\infty, z}=\k^1_{2k+3}$.
\end{proof}

As a corollary to \rlem{full limit is bounded}, we get the following.

\begin{corollary}
For every $z\in \mathbb{R}$, $\d_{\infty, z}=\k^1_{2k+3}$.
\end{corollary}

\subsection{The proof of the main theorem}

In this subsection, we work towards the proof of \rthm{main theorem}. Recall that
\begin{center}
$a_{2k+1, m}=\sup \{ \card{\leq^*} : \leq^*\in \utilde{\Gamma}_{2k+1, m}\}$\\
\end{center}
We let $\gg_{\infty, m, x}=\gg_{\infty, s_m, x}$ and $\gg^f_{\infty, m, x}=\gg^f_{\infty, s_m, x}$ and let
\begin{center}
$b_{2k+1, m}=\sup_{x\in \mathbb{R}}\gg_{\infty, m, x}$.
\end{center}
Notice that it follows from \rlem{full limit is bounded} that
\begin{center}
$b_{2k+1, m}=\sup_{x\in \mathbb{R}}\gg^f_{\infty, m, x}$.
\end{center}
It follows from \rthm{woodins thing} that
\begin{center}
$\k^1_{2k+3}=\sup_{m\in \omega} b_{2k+1, m}$.
\end{center}
To make the notation as simple as possible, we fix an odd integer $2k+1$. We will omit it from various subscripts from now until the end of this subsection. 

\begin{lemma}\label{sups}
$a_{2k+1, m}\leq b_{2k+1, m+1}$.
\end{lemma}
\begin{proof}
Fix $m\in \omega$ and let $\leq^*\in\utilde{\Gamma}_{2k+1, m}$. Let $z^*, \phi$ be such that for all $x, y\in \mathbb{R}$,
\begin{center}
$x\leq^* y \iff \M_{2k}(z^*, x, y) \models \phi[z^*, x, y, s_{m}]$.
\end{center}
Suppose towards a contradiction that $\card{\leq^*}=\sup_{x\in \mathbb{R}} \gg_{\infty, x, m+1}$ (this may produce another real parameter, but we assume that it is already part of $z^*$).

First notice that for every $l$, $\sup_{x\in \mathbb{R}} \gg_{\infty, l, x}< \k^1_{2k+3}$. This is because if $\sup_{x\in \mathbb{R}} \gg_{\infty, l, x}=\k^1_{2k+3}$ then because $\cf(\kappa^1_{2k+3})=\omega$ (see \cite{Moschovakis}), there must be $x$ such that $\gg_{\infty, l, x}=\k^1_{2k+3}$. But since $\d_{\infty, x}>\gg_{\infty, l, x}$, we get a contradiction. Thus, we can fix $z\in \mathbb{R}$ and $r\in \omega$ such that $z^*\leq_T z$ and $\gg_{\infty, r, z} >\sup_{x\in \mathbb{R}} \gg_{\infty, m, x}$.

Following Hjorth (see \cite{Hjorth01}), using Moschovakis' coding lemma (see \cite{Moschovakis}), we get $w\in \mathbb{R}$ and a $\Sigma^1_{2k+3}(w)$ set $B\subseteq \mathbb{R}^2$ such that $z\leq_T w$
\begin{enumerate}
\item if $(x, y)\in B$ then $x\in dom(\leq^*)$, $y\in dom(\leq_{z, r})$ and $\card{x}_{\leq^*}=\card{y}_{\leq_{z, r}}$,
\item for every $x\in dom(\leq^*)$ there is $y\in dom(\leq_{z,r})$ such that $(x, y)\in B$.
\end{enumerate}
Let $R$ be $\Pi^1_{2k+2}(w)$ such that $(x, y)\in B \iff \exists u R(w, x, y, u)$. We now construct an embedding of $\leq^*$ into $R_{w, s_{m+1}}$. Let $A=\{ (x, y , u) : R(w, x, y, u)\}$. Notice that whenever $a$ is a countable transitive set, $\leq^*\cap \mathbb{R}^{\M_{2k}(a)}\in \M_{2k}(a)$. We will abuse our notation and write $\leq^*$ for $\leq^*\cap \mathbb{R}^{\M_{2k}(a)}$.

Given a suitable $\P$, there is a maximal antichain $\mathcal{A}\subseteq \mathbb{B}^\P$ such that if $p\in \mathcal{A}$ then for some $\a$, in $\M_{2k}(\P)$
 \begin{enumerate}
\item $p\forces ``x_G=(x, y, u)\in A"$,
\item $p\forces``\forces_{Coll(\omega, \d^\P)} \card{x}_{\leq^*}=\a$".
\end{enumerate}
Notice that we can take $\mathcal{A}\in H_{m+1}^\P$. Let then $\mathcal{A}^\P$ be the least such maximal antichain. We can define $\leq^\P$ on $\mathcal{A}^\P$ as follows. Given $p\in \mathcal{A}$, let $\a_p$ be the ordinal $\a$ as in 2. Then for $p, q\in \mathcal{A}$, we let $p\leq^\P q$ iff $\a_p \leq \a_q$. Notice that $\card{\leq^\P}<\gg_{m+1}^\P$. The remaining part of the proof is similar to the proof of \rlem{full limit is bounded}.

Given now an $x\in dom(\leq^*)$, a suitable $\P$ and $p\in \mathcal{A}^\P$ we say $(\P, p)$ is $x$-stable if there is $(x, y, u)\in A$ which is generic over $\P$ for $\mathbb{B}^\P$ and
\begin{enumerate}
\item if $G\subseteq \mathbb{B}^\P$ is such that $x_G=(x, y, u)$ then $p\in G$,
\item whenever $(\R, q)$ is such that $\R$ is a correct iterate of $\P$ such that some $(x, y^*, u^*)\in A$ is generic over $\R$ for $\mathbb{B}^\R$, and $q\in \mathcal{A}^\R\cap H$ where $H\subseteq \mathbb{B}^\R$ is the $\R$-generic such that $x_H=(x, y^*, u^*)$, then
\begin{center}
$\card{q}_{\leq^\R}=_{\leq^\R}\card{\pi_{\P, \R, s_{m+1}}(p)}$.
\end{center}
\end{enumerate}

We claim that for every $x\in dom(\leq^*)$ there is $x$-stable $(\P, p)$. To see this, suppose not. First let $y, u$ be such that $(x, y, u)\in A$. Then let $\P$ be suitable such that $(x, y, u)$ is generic for $\mathbb{B}^\P$. There is then $p\in \mathcal{A}^\P$ such that if $G\subseteq \mathbb{B}^\P$ is $\P$-generic such that $x_G=(x, y, u)$ then $p\in G$. Let $\a=\a_{\P, p}$. Because $(\P, p)$ isn't $x$-stable we must have that there is a correct iterate $\R$ of $\P$ such that some $(x, y^*, u^*)\in A$ is generic over $\R$ for $\mathbb{B}^\R$, and if $H$ is the generic such that $x_H=(x, y^*, u^*)$ and $q\in H\cap \mathcal{A}^\R$ then
\begin{center}
$\card{q}\not=_{\leq^\R}\card{\pi_{\P, \R, s_{m+1}}(p)}$.
\end{center}
Let $y$ code $(\Q, \b)$ and let $y^*$ code $(\Q^*,\b^*)$. Notice that $(\Q, \b)=_{R_{z, r}}(\Q^*, \b^*)$.
Let also
\begin{center}
$i=i_{\P, \R}\rest \M_{\infty, z, \P}:\M_{\infty, z, \P} \rightarrow \M_{\infty, z,  \R}$.
\end{center}
We have that
\begin{center}
$i\circ \pi_{\Q, \infty, r, z, \P} : H_r^\P\rightarrow H_r^{\M_{\infty, z, \R}}$.
\end{center}
Because of Dodd-Jensen then we get that
\begin{center}
$i(\pi_{\Q, \infty, r, z, \P}(\b))\geq \pi_{\Q^*, \infty, r, z \R}(\b^*)$.
\end{center}
Notice that equality cannot hold. To see this, suppose $i(\pi_{\Q, \infty, r, z, \P}(\b))=\pi_{\Q^*, \infty, r, z, \R}(\b^*)$. We have that,
 \begin{center}
 $\M_{2k}(\P)\models p\forces $ ``if $(x_G)_2=(\dot{\Q}, \dot{\b})$ then $\pi_{\dot{\Q}, \infty, r, z, \P}(\dot{\b})=\check{\xi}$"\footnote{Here we think of a real $x$ as coding a triple $(x_1, x_2, x_3)$.}.
  \end{center}
 where $\check{\xi}=\pi_{\Q, \infty, r, z, \P}(\b)$. We then have by elementarity that there is $\R$-generic $H\subseteq \mathbb{B}^\R$ such that $i_{\P, \R}(p)\in H$ and if $(x_H)_2=(\S, \nu)$ then $\pi_{\S, \infty, r, z,  \R}(\nu)=i_{\P, \R}(\xi)$. But since we are assuming that $i(\pi_{\Q, \infty, r, z, \P}(\b))=\pi_{\Q^*, \infty, r, z, \R}(\b^*)$, we must have that $(\S, \nu)=_{R_{z, r}}(\Q^*, \b^*)$ and by the choice of $B$ we must have that $(x_H)_1=_{\leq^*} x$. This then implies that $i_{\P, \R}(p)=_{\leq^\R} q$, contradiction. Thus we must have that
\begin{center}
$i(\pi_{\Q, \infty, z, w\oplus\P, r}(\b))> \pi_{\Q^*, \infty, z, w\oplus\R, r}(\b^*)$.
\end{center}
Let then $\P_0=\P$, $(x, y_0, u_0)=(x, y, u)$, $\P_1=\R$ and $(x, y_1, u_1)=(x, y^*, u^*)$. Let $(\Q_0, \b_0)$ be the pair coded by $y_0$ and let $(\Q_1, \b_1)$ be the pair coded by $y_1$. Let $\xi_i=\pi_{\Q_i, \infty, z, r, \P_i}(\b_i)$ for $i=0,1$. It then follows from our discussion that $i_{\P_0, \P_1}(\xi_0)>\xi_1$.

By a repeated application of the argument used in the previous paragraph, we can get $\la \P_l, (\Q_l, \b_l), \xi_l : l\in \omega\ra$ such that
\begin{enumerate}
\item $\P_l\in \mathcal{F}_w$,
\item $\P_{l+1}$ is a correct iterate of $\P_l$,
\item $(\Q_l, \b_l) \in \mathcal{I}_{w, r}$ and $(Q_l, \b_l)$ is generic over $\P_l$ for $\mathbb{B}^{\P_l}$,
\item $\pi_{\Q_l, \infty, r, z, \P_l}(\b_l)=\xi_l$,
\item $i_{\P_l, \P_{l+1}}(\xi_l)>\xi_{l+1}$.
\end{enumerate}
Letting $\sigma_{l, j}:\P_l\rightarrow \P_j$ be the iteration embedding, letting $\Q$ be the direct limit of $\la \P_l, \sigma_{l, j} : l< j <\omega\ra$ and letting $\sigma_l:\P_l\rightarrow \Q$ be the iteration embedding we get that
$\la \sigma_l(\xi_l) : l <\omega\ra$ is a decreasing sequence of ordinals, contradiction. Thus, indeed, for every $x$ there is an $x$-stable $(\P, p)$.

Let then for each $x$, $S_x$ be the set of $x$-stable $(\P, p)$'s and let $\b_{\P, p}=\card{p}_{\leq^\P}$. Using uniformization, we can choose $(\P_x, p_x)\in S_x$. Notice now that
\begin{center}
 $x\leq^*y\iff (\P_x, p_x)\leq_{w, m} (\P_y, p_y)$.
\end{center}
(To see this, let $\R$ be a common iterate of $\P_x$ and $\P_y$ such that for some $u, v, u^*, v^*\in \mathbb{R}$, $(x, u, v)$ and $(y, u^*, v^*)$ are generic over $\R$ for $\mathbb{B}^\R$. Then by $x$ and $y$ stability, we must have that  $x\leq^*y$ holds if and only if $i_{\P_x, \R}(p_x)\leq^\R i_{\P_y, \R}(p_y)$.) We then have that $x\rightarrow (\P_x, p_x)$ is an order preserving map of $\leq^*$ into $R_{w, m+1}$. Therefore, $\card{\leq^*}\leq \card{R_{w, m+1}}\leq \sup_{x\in \mathbb{R}}\gg_{\infty, x, m+1}$, contradiction!
\end{proof}

We thus have that $\sup_{m\in \omega} a_{2k+1,m}\leq \k^1_{2k+3}$. Notice that for each $m\in \omega$ and $w\in \mathbb{R}$, $R_{w, m}\in \Gamma_{2k+1, m+1}(w)$. Because we have that $\sup_{m\in \omega}b_{2k+1,m}=\k^1_{2k+3}$, we easily get that $\sup_{m\in \omega} a_{2k+1, m} = \k^1_{2k+3}$. This then finishes the proof of the Main Theorem.

\section{Some remarks}

First of all, it turns out that $b_{2k+1, m}$ is a cardinal for every $m$ and moreover, $b_{2k+1, m}<\k^1_{2k+3}$. Here is the proof.

\begin{lemma}\label{as are cardinals} For every $m$, $b_{2k+1, m}<\k^1_{2k+3}$  and $b_{2k+1, m}$ is a cardinal.
\end{lemma}
\begin{proof}
We have that $b_{2k+1, m}< \k^1_{2k+3}$ because if for some $m$, $b_{2k+1, m}=\k^1_{2k+3}$ then because $\cf(\k^1_{2k+3})=\omega$, we can fix $x$ such that $\gg_{\infty, s_m, x}=\k^1_{2k+3}$. But this contradicts \rthm{woodins thing}. Thus, we have that $b_{2k+1, m}<\k^1_{2k+3}$.  Suppose no that for some $m$, $b_{2k+1, m}$ isn't a cardinal. Let $\k=\card{b_{2k+1, m}}$. Then $\k<b_{2k+1, m}$ and there is $A\subseteq \k$ such that $A$ codes a well-ordering of $\k$ of length $b_{2k+1, m}$. There is then a real $z$ such that $A\in \H_z$. We then can get $w$ such that $z\leq_T w$ and $\k<\gg^f_{\infty, s_m, w}$. It follows that $A\in \H_w$ and in particular, $\gg^f_{\infty, s_m, w}$ isn't a cardinal of $\H_w$. But clearly $\gg^f_{\infty, s_m, w}$ is a cardinal of $\H_w$, contradiction.
\end{proof}

We do not know if $a_{2k+1, m}=b_{2k+1, m}$. A more interesting question that comes up naturally is what is the exact place of $b_{2k+1, m}$ in the sequence of $\aleph$'s. We conjecture that $a_{2k+1, 0}=\d^1_{2k+2}$. One evidence for this is that by Hjorth's aforementioned result, $a_{1, 0}=u_2=\omega_2=\d^1_2$. More generally, Jackson showed that the sup of the lengths of $\Pi^1_{2k}$ prewellorderings is $\d^1_{2k}$ and $\Pi^1_{2k}$ is a subclass of $\Gamma_{2k+1, 0}$. The general question is open.

 It seems to be possible to use the directed system associated with $\M_{2n+1}$ to prove Kechris-Martin kind of results for $\Pi^1_{2k+3}$ (see \cite{KecMar}). In particular, one should be able to prove that $\Pi^1_{2k+2}$ is closed under quantification over $\k^1_{2k+3}$. Another application should be the uniqueness of $L[T_{2k}]$, i.e., it should be possible to prove, using ideas from this paper, that $L[T_{2k}]$ is independent of the choice of the scale that produces $T_{2k}$. This would generalize Hjorth's theorem on the uniqueness of $L[T_2]$ (see \cite{Hjorth96}).  It should also be possible to prove results like Solovay's $\Delta^1_3$-coding result (see \cite{Solovay}) for higher levels of projective hierarchy. The author, however, has no intuition on whether it is possible to use directed systems to carry out Jackson's analysis of projective ordinals. From an inner model theoretic point of view, Jackson's analysis remains a mystery.

\bibliography{directsystem}
\bibliographystyle{plain}
\end{document}